\documentclass{article}%
\usepackage{amsfonts}
\usepackage{amsmath}
\usepackage{amssymb}%
\usepackage{graphicx}
\providecommand{\U}[1]{\protect\rule{.1in}{.1in}}
\newtheorem{theorem}{Theorem}

\newtheorem{corollary}[theorem]{Corollary}

\newtheorem{definition}[theorem]{Definition}
\newtheorem{example}[theorem]{Example}

\newtheorem{lemma}[theorem]{Lemma}

\newtheorem{proposition}[theorem]{Proposition}
\newtheorem{remark}[theorem]{Remark}

\newenvironment{proof}[1][Proof]{\noindent\textbf{#1.} }{\ \rule{0.5em}{0.5em}}
\begin{document}

\title{Spectra related to the length spectrum}
\author{Conrad Plaut\\Department of Mathematics\\University of Tennessee\\Knoxville, TN 37996\\cplaut@math.utk.edu}
\maketitle

\begin{abstract}
We show how to extend the Covering Spectrum (CS) of Sormani-Wei to two
spectra, called the Extended Covering Spectrum (ECS) and Entourage Spectrum
(ES) that are new for Riemannian manifolds but defined with useful properties
on any metric on a Peano continuum. We do so by measuring in two different
ways the \textquotedblleft size\textquotedblright\ of a topological
generalization of the $\delta$-covers of Sormani-Wei called \textquotedblleft
entourage covers\textquotedblright. For Riemannian manifolds $M$ of dimension
at least 3, we characterize entourage covers as those covers corresponding to
the normal closures of finite subsets of $\pi_{1}(M)$. We show that
CS$\subset$ES$\subset$MLS and that for Riemannian manifolds these inclusions
may be strict, where MLS is the set of lengths of curves that are shortest in
their free homotopy classes. We give equivalent definitions for all of these
spectra that do not actually involve lengths of curves. Of particular interest
are resistance metrics on fractals for which there are no non-constant
rectifiable curves, but where there is a reasonable notion of Laplace Spectrum
(LaS). The paper opens new fronts for questions about the relationship between
LaS and subsets of the length spectrum for a range of spaces from Riemannian
manifolds to resistance metric spaces.


\end{abstract}

\section{Introduction}

The Length Spectrum (LS) of a Riemannian manifold is the set of lengths of
closed geodesics, with various notions of multiplicity. The notion goes back
at least to Huber's papers in the late 1950's (\cite{H1}, \cite{H2}) in which
the notion of LS is defined for what seems to be the first time. He showed
that for compact Riemann surfaces, LS and the Laplace Spectrum (LaS) determine
one another. Put another way, two Riemann surfaces have the same LS if and
only if they are isospectral. These results have been followed by a
decades-long investigation into the relationship between LS and LaS. We will
not give a detailed history here, but will mention the fundamental open
question of whether isospectral compact Riemannian manifolds have the same
Weak Length Spectrum (defined as LS ignoring multiplicities, also called the
Absolute Length Spectrum). For Riemann surfaces it is also known that LS and
LaS are each completely determined by a finite subset, the size of which is
bounded in the first case by the injectivity radius (Theorem 10.1.4,
\cite{Bu}), and in the second case by the injectivity radius and the genus
(Theorem 14.10.1, \cite{Bu}). This result suggests that there may be
interesting relationships among geometrically and topologically significant
subsets of LS, LaS, and other spectra from geometric analysis.

One of the most important subsets of LS is what Carolyn Gordon called the
[L]-spectrum (\cite{G}) and Christina Sormani and Guofang Wei called the
Minimum Length Spectrum (MLS) in \cite{SW1}. MLS is the set of lengths of
curves that are shortest in their free homotopy classes. As is well-known
(more on this later), in a compact Riemannian manifold there is always a
shortest curve in every free homotopy class and that curve must be a closed
geodesic, hence MLS $\subset$ LS. In \cite{G}, Gordon showed that there are
isospectral manifolds with distinct MLS, considering multiplicity as the
number of distinct free homotopy classes.

In \cite{SW1}, Sormani-Wei introduced a subset of $\frac{1}{2}$MLS called the
Covering Spectrum (CS). The existence of isospectral compact Riemannian
manifolds with different CS was established in 2010 by Bart de Smit, Ruth
Gornet and Craig Sutton (\cite{GSS3}, dimensions $\geq3$ and \cite{GSS2},
surfaces). Athough CS and MLS are not \textquotedblleft spectral
invariants\textquotedblright, these spectra have mathematical applications,
some of which we will mention below, and they may have nicer properties than
LS. For example MLS is discrete and CS is finite for any compact Riemannian
manifold, but LS may not be discrete (\cite{SS}).

In this paper we show how to extend CS to two spectra, called the Extended
Covering Spectrum (ECS) and the Entourage Spectrum (ES), that are new even for
compact Riemannian manifolds. ECS is discrete for arbitrary metrics on Peano
Continua (compact, connected, locally path connected spaces), but may contain
arbitrarily small values and so generally is not contained in $\frac{1}{2}$LS.
ES contains $2$CS and is contained in MLS, and therefore is discrete for
Riemannian manifolds, but we do not know whether it is discrete for arbitrary
compact geodesic spaces (metric spaces in which every $x,y$ are joined by a
curve, called a geodesic, of length equal to $d(x,y)$). We show by example
that for compact Riemannian manifolds, CS may be properly contained in ECS and
ES, and ES may be properly contained in MLS; that is, in general all of these
spectra are distinct. Counting multiplicity, the cardinality of ES, like MLS
but unlike CS, is a topological invariant, independent of any metric.

We also show how to extend the notion of MLS to compact geodesic spaces in
general, where, contrary to statements in \cite{G1} and \cite{G2}, there may
be no shortest curve in a free homotopy class, and when there is one it may
not be a closed geodesic (\cite{BPS}). For all of these length-type spectra
except LS itself, we give equivalent alternative definitions for geodesic
spaces that don't actually involve lengths of curves. That is, we define
\textquotedblleft length spectra\textquotedblright\ when there may be no
length. Of particular interest are resistance metrics on self-similar fractals
such as the Sierpinski Gasket and Carpet (which are Peano continua). These
metric spaces have a meaningful notion of Laplacian (see \cite{Ki}, \cite{Str}
for general references), but they are generally far from being geodesic
spaces. In fact there may be no non-constant rectifiable curves at all, and
hence empty LS. On the other hand, we show in \cite{PNL} that all of the
generalized spectra discussed in the present paper have infinitely many values
for these resistance metric spaces, hence they provide good proxies for
questions about the relationship between subsets of LS and LaS.

There are at least five characterizations of CS for compact geodesic spaces,
most of which we will use at some point in this paper, and we will add two
more. The original definition of CS due to Sormani-Wei (\cite{SW1}) begins
with a general construction of Spanier (\cite{Sp}) that produces a regular
covering map determined by an open cover $\mathcal{U}$ of a connected, locally
path connected space $X$. In $\pi_{1}(X)$, let $\Gamma(\mathcal{U}) $ be the
(normal) subgroup generated by the set of homotopy classes of loops of the
form $\overline{\alpha}\ast\lambda\ast\alpha$ where $\lambda$ lies entirely in
one of the open sets in $\mathcal{U}$ and $\alpha$ starts at the basepoint.
Here $\ast$ denotes concatenation and $\overline{\alpha}$ is the reverse
parameterization of $\alpha$. According to Spanier there is a regular covering
map of $X$ such that $\Gamma(\mathcal{U})$ is the image of $\pi_{1}(X)$ via
the homomorphism induced by the covering map. The deck group $\pi
_{1}(X)/\Gamma(\mathcal{U})$ of the covering map is in a sense a
\textquotedblleft fundamental group at the scale of $\mathcal{U}%
$\textquotedblright\ because it \textquotedblleft ignores small
holes\textquotedblright\ contained in elements of $\mathcal{U}$ when modding
out by $\Gamma(\mathcal{U})$. Sormani-Wei took for $\mathcal{U}$ the open
cover of $X$ by (open) $\delta$-balls $B(x,\delta)$, and called the resulting
covering map the $\delta$-cover of $X$. In the case when $X$ is a compact
geodesic space, Sormani-Wei showed that the equivalence type of $\delta
$-covers changes at certain values, which they called the Covering Spectrum,
that are discrete in $(0,\infty)$. Viewing $\delta$ as a parameter, as
$\delta$ shrinks from the diameter of $X$ to $0$, $X$ \textquotedblleft
unrolls\textquotedblright\ more and more at the discrete values in CS. If $X$
has a universal cover $\widetilde{X}$, then $\widetilde{X}$ is the $\delta
$-cover for all sufficiently small $\delta>0$ and CS is finite. By
\textquotedblleft universal cover\textquotedblright\ we mean in a categorical
sense, which for compact geodesic spaces is equivalent to finiteness of CS
(see Theorem 3.4 in \cite{SW1} and \cite{W} for related equivalent conditions).

Another way to characterize CS uses the discrete homotopy methods of
Berestovskii-Plaut, developed in 2001 for topological groups (\cite{BPTG})
then uniform spaces in 2007 (\cite{BPUU}). In 2010, Plaut-Wilkins (\cite{PW1})
focused on the special case of metric spaces, where discrete homotopy theory
means replacing continuous curves and homotopies by discrete sequences and
homotopies called $\varepsilon$-chains and $\varepsilon$-homotopies,
respectively. An $\varepsilon$-chain is a finite sequence of points
$\{x_{0},...,x_{n}\}$ in a metric space such that for all $i$, $d(x_{i}%
,x_{i+1})<\varepsilon$. Discrete homotopies consist of finitely many steps
adding or removing a single point in an $\varepsilon$-chain (fixing the
endpoints) so that the sequence remains an $\varepsilon$-chain at each step.
Discrete homotopies \textquotedblleft ignore small holes\textquotedblright%
\ simply by skipping over them. As in \cite{BPTG} and \cite{BPUU}, one can
imitate the classical construction of the universal covering space,
substituting $\varepsilon$-chains for curves and $\varepsilon$-homotopies for
homotopies. This produces what Plaut-Wilkins called $\varepsilon$-covers
$\phi_{\varepsilon}:X_{\varepsilon}\rightarrow X$. The fact that both ways of
\textquotedblleft ignoring small holes\textquotedblright\ are essentially the
same (despite the very different constructions) was shown by Plaut-Wilkins in
\cite{PW2}: for compact geodesic spaces, the Sormani-Wei $\delta$-covers are
equivalent to the Plaut-Wilkins $\varepsilon$-covers when $\varepsilon
=\frac{2}{3}\delta$. Plaut-Wilkins also defined the Homotopy Critical Spectrum
(HCS) in \cite{PW1} to be the set of all $\varepsilon$ such that there is an
$\varepsilon$-loop that is not $\varepsilon$-null (homotopic) but is $\delta
$-null when considered as a $\delta$-chain for any $\delta>\varepsilon$. They
also showed in \cite{PW2} that CS $=\frac{3}{2}$HCS for compact geodesic
spaces. In this paper, since we will refer frequently to \cite{PW1} and
\cite{PW2}, we will generally use the notation of $\varepsilon$-covers.
Plaut-Wilkins also defined special closed geodesics called \textquotedblleft
essential circles\textquotedblright\ whose lengths are precisely three times
the values of HCS, discussed in more detail later in this paper.

ECS, ES, and our generalized definition of MLS involve expanding the class of
$\varepsilon$-covers to a larger class of covering spaces called
\textit{entourage covers. }Entourage covers are defined using the original
construction of Berestovskii-Plaut for uniform spaces (\cite{BPUU}). The
present paper is written so that no special knowledge of uniform spaces is
required, and but the language and framework of uniform spaces are useful. If
$E$ is an \textquotedblleft entourage\textquotedblright\ in a uniform space
$X$, which is a special symmetric set containing an open subset of the
diagonal in $X\times X$, one may define \textquotedblleft$E$%
-chains\textquotedblright\ to be sequences $\{x_{0},...,x_{n}\}$ such that for
all $i$, $(x_{i},x_{i+1})\in E$. Then $E$-homotopies and the corresponding
covering map $\phi_{E}:X_{E}\rightarrow X$ may be defined analogous to
$\varepsilon$-homotopies and $\varepsilon$-covers. We will provide more
details in the next section. Compact topological spaces have a unique uniform
structure in which entourages are just any symmetric subsets of $X\times X$
containing an open set containing the diagonal, and therefore the maps
$\phi_{E}$ are determined only by the topology (whereas $\varepsilon$-covers
are determined by the metric).

In metric spaces there are \textit{metric entourages} $E_{\varepsilon
}:=\{(x,y):d(x,y)<\varepsilon\}$ for $\varepsilon>0$. $E_{\varepsilon}$-chains
and $E_{\varepsilon}$-homotopies are precisely the $\varepsilon$-chains and
$\varepsilon$-homotopies previously described. Metric entourages form a
\textit{basis} for a uniform structure uniquely determined by the metric
(although there may be other uniform structures compatible with the topology).
That is, entourages in a metric space are simply symmetric subsets of $X\times
X$ that contain some $E_{\varepsilon}$. We also use an analogous notation for
$E$-balls: $B(x,E):=\{y\in X:(x,y)\in E\}$. If an $E$-loop is $E$-homotopic to
the trivial chain, we say it is $E$\textit{-null}.

In general the covering maps $\phi_{E}$ can be problematic, especially if the
balls $B(x,E)$ are not connected. For example, if $X$ is a compact metric
space that is not geodesic, the $\varepsilon$-covers many have infinitely many
components and HCS may not only not be discrete, it may even be dense in
$[0,1]$ (\cite{Cetal}, \cite{WD}). Disconnected metric balls also may occur in
resistance metrics on finite graphs (\cite{Ba}, Remark 7.19) and have been
numerically verified by Cucuringu-Strichartz for certain resistance metrics on
the Sierpinski Gasket (\cite{CS}, Section 4). For non-geodesic spaces, it is
not clear whether focusing on metric entourages is likely to be a successful
strategy. In some sense, these problems occur because there is a disconnect
between the metric and the underlying topology, which is improperly
\textquotedblleft viewed\textquotedblright\ by metric balls.

We address this problem by restricting attention to what we call
\textquotedblleft chained entourages\textquotedblright: an entourage $E$ is
\textit{chained}\ if it is contained in the closure of its interior, which is
assumed compact if the space is not, and whenever $(x,y)\in E$, $x,y$ may be
joined by an $F$-chain that lies entirely in $B(x,E)\cap B(y,E)$, for any
entourage $F$. That is, $x$ and $y$ may be joined in $B(x,E)\cap B(y,E)$ by
\textquotedblleft arbitrarily fine\textquotedblright\ chains.
\textit{Entourage covers} are by definition those covers $\phi_{E}$ such that
$E$ is a chained entourage; we will sometimes call it the $E$-cover of $X$. In
a geodesic space, metric entourages are always chained. This is true because
by the triangle inequality, any geodesic joining $x,y$ must stay inside
$B(x,\varepsilon)\cap B(y,\varepsilon)$. Then one may simply take arbitrarily
fine chains along the geodesic.

One advantage of discrete methods is that they are amenable to counting
arguments. For example, while Sormani-Wei showed using convergence methods
that the size of CS is bounded in any Gromov-Hausdorff precompact class, in
\cite{PW1} we actually give an explicit bound. Moreover, since for compact
Riemannian manifolds $\pi_{1}(M)=\pi_{\varepsilon}(E)$ for all sufficiently
small $\varepsilon>0$, by building a simplicial model of the space,
Plaut-Wilkins were able to give an explicit fundamental group finiteness
theorem generalizing those of Anderson (\cite{An}) and Shen-Wei (\cite{ShW}).
Similarly, we are able to prove the following explicit finiteness theorem,
where $C(X,\varepsilon)$ denotes the number of $\varepsilon$-balls needed to
cover $X$ and $\sigma(E):=\sup\{\varepsilon:E_{\varepsilon}\subset E\}$ (which
is a measure of \textquotedblleft size\textquotedblright\ of $E$).

\begin{theorem}
\label{T2}Let $X$ be a compact geodesic space and $\varepsilon>0$. Then the
number $NC(\varepsilon)$ of equivalence classes of $E$-covers $\phi_{E}%
:X_{E}\rightarrow X$ such that $\sigma(E)\geq\varepsilon$ is at most
\[
2^{C(X,\frac{\varepsilon}{4})^{40C(X,\frac{\varepsilon}{2})}}\text{.}%
\]

\end{theorem}

In order to apply the above theorem in more generality, recall that the
Bing-Moise Theorem (\cite{B}, \cite{M}) says, in modern terminology, that
every Peano continuum has a compatible geodesic metric. \textquotedblleft
Compatible\textquotedblright\ means precisely that every metric entourage in
the original metric contains a metric entourage in the geodesic metric, and
vice versa. We immediately obtain:

\begin{corollary}
\label{C1}If $X$ is a Peano Continuum with a given (possibly non-geodesic)
metric and $\varepsilon>0$, then $NC(\varepsilon)<\infty$.
\end{corollary}

We are now in a position to modify the original Sormani-Wei definition of CS
to apply to entourage covers. The only complication is that $E$-covers, unlike
$\varepsilon$-covers, are not totally ordered by the relation of one space
covering another. But by Corollary \ref{C1} there are certain discrete values
of $\varepsilon$ such that $NC(\varepsilon)$ \textit{strictly increases}, and
we define ECS to be those values. We may also definte the \textit{multiplicity
}of a value $\varepsilon$ in ECS to be $NC(\varepsilon)-NC(\delta)$ for
$\delta<\varepsilon$ sufficiently close to $\varepsilon$. With a little more
effort we show:

\begin{theorem}
\label{csecs}If $X$ is a compact geodesic space then CS $\subset$ ECS.
\end{theorem}

An immediate consequence of Corollary \ref{C1} is that ECS is discrete in
$(0,\infty)$ for any metric on a Peano continuum.

\begin{remark}
By Gromov's Precompactness Criterion (\cite{G1}, \cite{G2}), a corollary of
Theorem \ref{T2} is that $NC(\varepsilon)$, and hence the number of elements
of ECS greater than $\varepsilon$, is uniformly bounded below for any fixed
$\varepsilon$ in any Gromov-Hausdorff precompact class of compact geodesic
spaces. This extends the corresponding statement about CS.
\end{remark}

We aleady know that any $\varepsilon$-cover for any geodesic metric is an
entourage cover, and a natural question is: in general, which regular covers
are entourage covers? We have the following necessary algebraic condition. The
meaning of \textquotedblleft covering map corresponding to $N$%
\textquotedblright\ is standard from algebraic topology and will be reviewed
as part of the proof.

\begin{theorem}
\label{obstruct}Let $X$ be a semi-locally simply connected Peano continuum and
$N\subset\pi_{1}(X)$ be a normal subgroup. If the covering map corresponding
to $N$ is an entourage cover, then $N$ is the normal closure of a finite set.
\end{theorem}

\begin{remark}
\label{finpres}The normal closure of a subset of a group is by definition the
smallest normal subgroup containing it. Being the normal closure of a finite
set is intimately connected with the study of finitely presented groups, since
a quotient of a finitely presented group is finitely presented if and only if
the kernel is the normal closure of a finite set. We are unable to find a
reference for this equivalence, although the proof of one implication may be
found in \cite{Gu}, Lemma 3, which cites Siebenmann's dissertation for the
statement. The other implication is an exercise in basic algebra using the
formal definition of \textquotedblleft finitely presented\textquotedblright.
\end{remark}

\begin{remark}
Suppose that $G$ is a finitely presented group with a quotient that is not
finitely presented (such $G$ are well-known to exist). Then as is also
classically known, one may construct a compact $4$-manifold $M$ with $G$ as
its fundamental group. Therefore $M$ must have a regular cover that is not an
entourage cover.
\end{remark}

For manifolds of dimension at least $3$, the condition in Theorem
\ref{obstruct} is also sufficient:

\begin{theorem}
\label{TD3}Let $M$ be any compact smooth manifold of dimension at least $3$.
If $G\subset\pi_{1}(M)$ is the normal closure of a finite set then $M$ has a
Riemannian metric for which the cover corresponding to $G$ is an $\varepsilon
$-cover for the metric.
\end{theorem}

\begin{corollary}
\label{TD3C}If $M$ is a smooth manifold of dimension at least $3$, a normal
subgroup $G$ of $\pi_{1}(X)$ corresponds to an $E$-cover if and only if $G$ is
the normal closure of a finite set. Moreover, we may always take $E$ to be an
open entourage.
\end{corollary}

For closed manifolds of dimension $1$ there is a simple statement: the only
entourage covers of the circle are the trivial and universal covers (which are
always entourage covers for manifolds), see Example \ref{circle}. We do not
know much about the situation for closed surfaces, including, for example,
exactly which covers of the $2$-torus are entourage covers. See also Example
\ref{moebius} concerning the M\"{o}bius Band.

We will define later in this paper (Definition \ref{ehomc}) a notion of
$E$-homotopy (or free $E$-homotopy) of curves, which essentially means the two
curves may be \textquotedblleft$E$-subdivided\textquotedblright\ into
$E$-chains that are $E$-homotopic (or freely $E$-homotopic). Equivalently, the
curve lifts as a loop to $X_{E}$ (Lemma \ref{ehomo}). We say a chained
entourage is \textquotedblleft essential\textquotedblright\ if there is an $E
$-loop that is not $E$-null. The next proposition (correctly) generalizes the
classical statement mentioned above about shortest curves in free homotopy
classes in Riemannian manifolds. We will use the term \textquotedblleft%
$\varepsilon$-geodesic\textquotedblright\ to describe a curve that is
arclength parameterized and minimizing on all intervals of length
$\varepsilon$, i.e. the distance between the endpoints of such segments is
$\varepsilon$. This is similar to the concept of $\frac{1}{k}$-geodesic of
Sormani (\cite{SCL}) but we obtain more precise results by allowing arbitrary
values of $\varepsilon$. A closed curve $c$ that is an $\varepsilon$-geodesic
for some $\varepsilon>0$ for any reparameterization involving a parameter
shift is called a \textquotedblleft closed $\varepsilon$%
-geodesic\textquotedblright\ or simply a \textquotedblleft closed
geodesic\textquotedblright\ if a particular $\varepsilon$ isn't specified.
Briefly, we express this by saying $c$ is minimal on all segments of length
$\varepsilon$, understanding that when $c$ is closed this includes segments
that have the common start/end point in the interior.

\begin{proposition}
\label{semi}Let $X$ be a compact geodesic space, $E$ be a chained entourage in
$X$, and $c$ be a closed curve in $X$. Then $c$ has a shortest curve
$\overline{c}$ in its free $E$-homotopy class, and for any such $\overline{c}$,

\begin{enumerate}
\item $\overline{c}$ is non-constant if and only if $E$ is essential, and

\item if $\overline{c}$ is non-constant then $\overline{c}$ is a closed
$\frac{3\varepsilon}{2}$ geodesic whenever $E_{\varepsilon}\subset E.$
\end{enumerate}
\end{proposition}

We now have the following characterizations of MLS and CS for Riemannian
manifolds. Each pair consists of a statement involving lengths of curves
(which in a general metric space might be infinite), and another that makes
sense and is always finite in any metric space. By the length $L(\alpha)$ of a
finite chain $\alpha=\{x_{0},...x_{n}\}$ in a metric space we mean $\sum
_{i=1}^{n}d(x_{i-1},x_{i})$.

\begin{theorem}
\label{final}Suppose that $M$ is a compact Riemannian manifold. Then over the
set of all essential entourages $E$,

\begin{enumerate}
\item MLS is the set of

\begin{enumerate}
\item lengths of non-constant closed curves that are shortest in their free
$E$-homotopy class.

\item lengths of non-trivial $E$-loops that are shortest in their free
$E$-homotopy class.
\end{enumerate}

\item CS is the set of

\begin{enumerate}
\item half the shortest lengths of closed curves that are not freely $E$-null.

\item half the shortest lengths of $E$-loops that are not freely $E$-null.
\end{enumerate}
\end{enumerate}
\end{theorem}

We may now simply use Theorem \ref{final}.1.a as the \textit{definition} of
MLS for compact geodesic spaces, and with this definition Theorem \ref{final}
is true for any compact geodesic space (Theorem \ref{final'}). We may use
Theorem \ref{final}.1b and Theorem \ref{final}.2b as the definitions for
arbitrary metric spaces, although if the underlying space is not a Peano
continuum then there may not be many (or any!) essential entourages.

\begin{remark}
\label{twodefs}We do not know whether, for arbitrary metric spaces, the
definition using Theorem \ref{final}.1b is equivalent to the definition of CS
for metric spaces given in \cite{GSS3}, namely that CS is the collection of
all $\varepsilon>0$ such that some covering map is maximally evenly covered on
all $\varepsilon$-balls.
\end{remark}

While CS $=$ $\frac{3}{2}$HCS for compact geodesic spaces, the approaches of
Sormani-Wei and Plaut-Wilkins diverge when more general $E$-covers are added
to the mix. The Sormani-Wei playbook leads to ECS as we have already
described. To apply the Plaut-Wilkins approach, for an open, chained entourage
$E$ in a metric space $X$ we define an $E$-loop $\lambda$ (resp. curve loop
$c$) to be $E$\textit{-critical} if $\lambda$ (resp. $c$) is not $E$-null but
is $\overline{E}$-null, where $\overline{E}$ is the closure of $E$. If $E$ has
a critical $E$-loop then we will say that $E$ is \textit{critical}, and we let
$\psi(E):=\inf\{L(\lambda):\lambda$ is an $E$-critical $E$-loop$\}$. We define
the Entourage Spectrum ES to be the set of $\psi(E)$ for all critical
entourages $E$. Note that this definition does not involve lengths of curves.
We show:

\begin{theorem}
\label{final2}Let $X$ be a compact geodesic space. Then

\begin{enumerate}
\item For every open, chained entourage $E$ there is a critical $E$-loop
$\lambda$ if and only if there is a critical loop $c$.

\item If $E$ is a critical entourage then there is a critical $E$-loop (resp.
critical loop $c$) such that $L(c)=L(\lambda)=\psi(E)$.

\item $2$CS $=$ $3$HCS $\subset$ ES $\subset$ MLS.
\end{enumerate}

Moreover, there are compact Riemannian manifolds for which both of the above
inclusions are proper. There also are pairs of diffeomorphic Riemannian
manifolds that have the same CS but different ES, and pairs that have the same
ES but different MLS.
\end{theorem}

\begin{remark}
The proof of Theorem \ref{final2} uses in essential ways the fact that the
metric is geodesic--in particular by lifting the metric to a geodesic metric
on $X_{E}$. However, the \textquotedblleft lifted metric\textquotedblright%
\ defined later in this paper is defined for any metric, and in particular
some of these methods may be modified for arbitrary metrics on Peano continua
(\cite{PNL}).
\end{remark}

\begin{remark}
It seems there are interesting questions involving these metric invariants
akin to classical results in the metric geometry of Riemannian manifolds. For
example, suppose that $M$ is a compact Riemannian manifold. Are there metrics
having a particular fixed spectrum (pick one of CS, ECS, ES, MLS) that are
optimal with respect to some other geometric parameters, for example having
minimal volume with fixed bounds on sectional curvature? And if so what is the
regularity of those optimal metrics? In effect this question fixes the size of
certain significant \textquotedblleft holes\textquotedblright\ in the space
and asks how the space can minimally be stretched to maintain those sizes
while constraining curvature.
\end{remark}

\begin{remark}
\label{different}Sormani-Wei have explored ways in which to extend ideas
related to CS to non-compact spaces (\cite{SWCUT}, \cite{SWNC}). Certainly CS
has meaning for non-compact geodesic spaces, for example, but it will
\textquotedblleft miss\textquotedblright\ any loops that are homotopic to
arbitrarily small loops, for example in the surface obtained by revolving the
graph of $y=e^{x}$ around the $x$-axis. Along these lines, we note that some
of the basic and technical results in this paper only require a kind of
uniform local compactness (Remark \ref{proper})--which is why we state in the
definition of \textquotedblleft chained\textquotedblright\ that the closure of
the entourage is compact when the space is not. An alternative approach to
understand such spaces might be to consider all possible uniform structures
that are compatible with the underlying topology of a given metric on the
space, rather than just the uniform structure induced by the particular
metric. For example, for the surface mentioned above the uniform structure
compatible with the metric does not \textquotedblleft see\textquotedblright%
\ the fundamental group but the uniform structure of the same space metrized
as a flat cylinder has the universal cover as an $\varepsilon$-cover.
\end{remark}

\section{Basic Constructions}

This section has a mixture of background from \cite{BPUU}, extensions of some
results in \cite{PW1}, and a completely new basic result called the Ball
Continuity Lemma. The length of a curve is defined in the standard way for
metric spaces and it is a classical result that curves having finite length
(i.e. rectifiable curves) in metric spaces always have monotone
reparameterizations proportional to arclength. We will always assume
rectifiable curves are parameterized this way. See \cite{PS} for a review,
with references, of many basic concepts from metric geometry.

We now recall a bit of basic terminology for uniform spaces. One should keep
in mind the two fundamental examples mentioned in the Introduction: metric
spaces and compact topological spaces. We have already defined the metric
entourage $E_{\varepsilon}$ in a metric space $X$. In general, a
\textit{uniform structure} on an (always Hausdorff) topological space $X$ is a
collection of symmetric subsets of $X\times X$ that contain an open set
containing the diagonal, which are called \textit{entourages.} Moreover,
entourages have the following properties: (UA) Their intersection is the
diagonal (equivalent to Hausdorff), and (UB) for every entourage $E$ there
exists an entourage $F$ such that
\[
F^{2}:=\{(x,z):\text{for some }y,(x,y),(y,z)\in F\}
\]
is contained in $E$. For metric entourages we note that it follows from the
triangle inequality that $\left(  E_{\frac{\varepsilon}{2}}\right)
^{2}\subset E_{\varepsilon}$. We may also iteratively define, for any
entourage $F$, $F^{n}$. Equivalently, in the terminology from the
Introduction, $F^{n}$ consists of all $(x,y)\in X\times X$ such that there is
an $F$-chain $\{x=x_{0},...,x_{n}=y\}$.

As mentioned in the Introduction, for an entourage $E$ in a uniform space $X$,
an $E$-chain consists of a finite sequence $\alpha=\{x_{0},...,x_{n}\}$ in $X$
such that for all $i$, $\left(  x_{i},x_{i+1}\right)  \in E$. We define
$\nu(\alpha)=n.$ The concatenation of two chains $\alpha=\{x_{0},...,x_{n}\}$
and $\beta=\{y_{0}=x_{n},y_{1},...,y_{m}\}$ is the chain $\alpha\ast
\beta:=\{x_{0},...,x_{n}=y_{0},y_{1},...,y_{m}\}$ and the reversal of $\alpha$
is the chain $\overline{\alpha}=\{x_{n},...,x_{0}\}$. An $E$-homotopy between
$E$-chains $\alpha$ and $\beta$ consists of a finite sequence $\{\alpha
=\eta_{0},...,\eta_{n}=\beta\}$ of $E$-chains $\eta_{i}$ all having the same
endpoints such that for all $i$, $\eta_{i}$ differs from $\eta_{i+1}$ by one
of the following two \textit{basic moves}:

(1) \textit{Insert} a point $x$ between $x_{i}$ and $x_{i+1}$, which we will
denote by
\[
\{x_{0},...,x_{i},\overbrace{x},x_{i+1},...,x_{n}\}
\]
and which is \textquotedblleft legal\textquotedblright\ provided $(x_{i},x)\in
E$ and $(x,x_{i+1})\in E$.

(2) \textit{Remove} a point $x_{i}$ (but never an endpoint!), which we will
denote by
\[
\{x_{0},...,x_{i-1},\underbrace{x_{i}},x_{i+1},...,x_{n}\}
\]
and which is legal provided $(x_{i-1},x_{i+1})\in E$.

The $E$-homotopy equivalence class of an $E$-chain $\alpha$ is denoted by
$[\alpha]_{E}$. We will sometimes abuse notation by dropping brackets, for
example writing $[x_{0},...,x_{n}]_{E}$ rather than $[\{x_{0},...,x_{n}%
\}]_{E}$. We note that if $E\subset F$ then $\alpha$ may also be considered as
an $F$-chain, and $[\alpha]_{F}$ also makes sense. Fixing a basepoint $\ast$,
the collection of all $[\alpha]_{E}$ such that the first point of $\alpha$ is
$\ast$ is called $X_{E}$. For any entourage $F\subset E$ in $X$, we define
$F^{\ast}\subset X_{E}\times X_{E}$ to be the set of all ordered pairs
$([\alpha]_{E},[\beta]_{E})$ such that $[\overline{\alpha}\ast\beta
]_{E}=[x_{n},y_{m}]_{E}$, where $x_{n},y_{m}$ are the endpoints of
$\alpha,\beta$, respectively, and $(x_{n},y_{m})\in F$. The sets $F^{\ast}$
form (a basis of) a uniform structure on $X_{E}$, which we will call the
\textquotedblleft lifted uniform structure\textquotedblright.

\begin{remark}
\label{star}It is an easy exercise, worthwhile for the unfamiliar reader, to
check that this definition of $F^{\ast}$ is equivalent to the more cumbersome
but sometimes useful original from \cite{BPUU}, namely that up to $E$-homotopy
we may write $\alpha=\{\ast=x_{0},...,x_{n-1},x_{n}\}$ and $\beta=\{\ast
=x_{0},...,x_{n-1},y_{n}\}$ with $(x_{n},y_{n})\in F$.
\end{remark}

Since $E$-homotopies don't change endpoints, the endpoint map $\phi_{E}%
:X_{E}\rightarrow X$, $\phi_{E}([x_{0},...,x_{n}]_{E}):=x_{n}$ is
well-defined, and its restriction to any $E^{\ast}$-ball $B([\alpha
]_{E},E^{\ast})$ is a bijection onto its image $B(x,E)$ under $\phi_{E}$,
where $x$ is the endpoint of $\alpha$. Since $\phi_{E}$ is a local bijection,
$X_{E}$ has a unique topology such that $\phi_{E}$ is a local homeomorphism,
with a basis given by all $F^{\ast}$-balls with $F\subset E$. This topology is
compatible with the lifted uniform structure.

Concatenation is compatible with $E$-homotopies--that is, if $\alpha_{1}%
,\beta_{1}$ are $E$-homotopic to $\alpha_{2},\beta_{2}$, respectively, then
$\left[  \alpha_{2}\ast\beta_{2}\right]  _{E}=\left[  \alpha_{1}\ast\beta
_{1}\right]  _{E}$. Concatenation induces a group structure on the set of all
$E$-homotopy classes $\pi_{E}(X)$ of $E$-loops starting and ending at the
basepoint $\ast$. That is, $[\lambda_{1}]_{E}[\lambda_{2}]_{E}=[\lambda
_{1}\ast\lambda_{2}]_{E}$ and $[\lambda]_{E}^{-1}=[\overline{\lambda}]_{E}$,
with $[\ast]_{E}$ as identity. The group $\pi_{E}(X)$ acts on $X_{E}$ induced
by pre-concatenation of any loop to an $E$-chain starting at $\ast$, and the
resulting maps are uniform homeomorphisms (i.e. preserve the uniform
structure). With this topology, $\phi_{E}$ is a regular covering map with deck
group $\pi_{E}(X)$ such that the $E$-balls are evenly covered by disjoint
unions of $E^{\ast}$-balls. Moreover, $X$ is identified with the quotient
space $X_{E}/\pi_{E}(X)$. Two $E$-loops $\lambda_{1}$ and $\lambda_{2}$ are
said to be \textit{freely }$E$-homotopic if there exist $E$-chains $\alpha$
and $\beta$ starting at a common point $x_{0}$, to the initial points of
$\lambda_{1}$ and $\lambda_{2}$, respectively, such that
\begin{equation}
\overline{\alpha}\ast\lambda_{1}\ast\alpha\text{ is }E\text{-homotopic to
}\overline{\beta}\ast\lambda_{2}\ast\beta\text{.}\label{tired}%
\end{equation}
It is easy to check, and we will use without reference, the following facts:
If we can satisfy Formula \ref{tired} for some $x_{0}$ then we can do it for
any other point, including the basepoint $\ast$. Likewise $\lambda_{1}$ and
$\lambda_{2}$ are freely $E$-homotopic if and only if given an $E$-chain
$\alpha$ from $x_{0}$ to the initial point of $\lambda_{1}$ then we can always
find a $\beta$ so that Formula \ref{tired} is satisfied.

\begin{remark}
One can equivalently define free $E$-homotopies of $E$-loops by adding a
special \textquotedblleft homotopy step\textquotedblright\ to enable one to
move the common start/end point. That is, \textquotedblleft double
it\textquotedblright\ by adding a repeat of the start/end point, then add a
new point (which becomes the new start/end point) between the doubled points.
However, this does not seem simpler in all of its details, or more useful than
the definition we have given.
\end{remark}

Whenever $E\subset F$ there is a natural covering map $\phi_{FE}%
:X_{E}\rightarrow X_{F}$ that simply treats an $E$-chain as an $F$-chain. That
is, $\phi_{FE}([\alpha]_{E})=[\alpha]_{F}$. Given $D\subset E\subset F$, by
definition $\phi_{FD}=\phi_{FE}\circ\phi_{ED}$. The restriction of $\phi_{EF}$
to $\pi_{F}(X)$ is a homomorphism denoted by $\theta_{EF}:\pi_{F}%
(X)\rightarrow\pi_{E}(X)$. This homomorphism is injective (resp. surjective)
if and only if $\phi_{EF}$ is injective (resp. surjective), and plays a
critical role in this paper. Note that the mapping $\phi_{EF}$ may be
identified with the quotient mapping from $X_{F}$ to $X_{F}/\ker\theta
_{EF}=X_{E}$. In the special case of metric entourages, we denote
$\phi_{E_{\varepsilon}E_{\delta}}$ by $\phi_{\varepsilon\delta}$ as in
\cite{PW1}.

We will need the following general lemma, which partly justifies why we
require that the balls in a chained entourage be contained in the closure of
their interior.

\begin{lemma}
(Ball Continuity) \label{bc}Suppose $E$ is an entourage in a metric space $X$
and $x_{i}\rightarrow x$ in $X$. Then

\begin{enumerate}
\item If $y_{i}\rightarrow y$ in $X$ and $(x_{i},y_{i})\in E$ for all $i$,
then $(x,y)\in\overline{E}$.

\item If $E$ has compact closure then $B(x_{i},E)$ is Hausdorff convergent to
a subset of $B(x,\overline{E})$.

\item If $E$ is open then $B(x,E)$ is the union of the sets $B(x,E)\cap
B(x_{i},E)$.

\item If $E$ is contained in the compact closure of its interior then
$B(x_{i},E)\underset{H}{\rightarrow}B(x,E)$ (meaning convergence in the
Hausdorff metric).
\end{enumerate}
\end{lemma}

\begin{proof}
The first part is true in any uniform space: just note that a if $U$ and $V$
are open sets containing $x$ and $y$, respectively, then $U\times V$
eventually contains some $(x_{i},y_{i})\in E$.

For the second part, let $A$ be the (compact) set of all limits of convergent
sequences $(z_{i})$ such that $z_{i}\in B(x_{i},E)$. By the first part,
$A\subset B(x,\overline{E})$. For any $\varepsilon>0$, we may cover $A $ by
finitely many balls $B(p_{1},\frac{\varepsilon}{2}),...,B(p_{m},\frac
{\varepsilon}{2})$. Then for sufficiently large $k$ there are points
$p_{j}^{\prime}\in B(x_{k},E)$ such that $d(p_{j},p_{j}^{\prime}%
)<\frac{\varepsilon}{2}$, and from the triangle inequality it follows that the
$\varepsilon$-neighborhood of $B(x_{k},E)$ contains $A$. Now suppose that for
all $k$, the $\varepsilon$-neighborhood of $A$ does not contain $B(x_{k},E)$.
That is, for all $k$ there exist $w_{k}\in B(x_{k},E)$ such that
$d(w_{k},z)\geq\varepsilon$ for all $z\in A$. Since $(w_{k},x_{k})\in E$ and
$E$ has compact closure, by taking a subsequence if necessary we may assume
that $(w_{k},x_{k})$ is convergent to $(w,x)\in\overline{E}$. By definition,
$w\in A$, so $d(w_{k},w)\geq\varepsilon$ for all $k$, a contradiction to
$w_{k}\rightarrow w$.

For the third part, let $y\in B(x,E)$. Since $E$ is open there exist open
$U,V$ in $X$ such that $(x,y)\in U\times V\subset E$. Then for large enough $i
$, $x_{i}\in U$ and therefore $(x_{i},y)\in U\times V\subset E$, i.e. $y\in
B(x_{i},E)$.

Finally, in the special case when $E$ is open with compact closure, the fourth
part follows from the second and third parts. The proof in general is now
finished by observing that in general if $A_{i}\underset{H}{\rightarrow}A$
then any sequence of dense subsets of $A_{i}$ is Hausdorff convergent to any
dense subset of $A$.
\end{proof}

We will now extend the idea of the \textquotedblleft lifted
metric\textquotedblright\ from \cite{PW1} to this more general situation. The
following definition is the same as the corresponding part of Definition 12,
\cite{PW1}, with \textquotedblleft$E$\textquotedblright\ replacing
\textquotedblleft$\varepsilon$\textquotedblright:

\begin{definition}
\label{metdef}Let $X$ be a metric space and $[\alpha]_{E},[\beta]_{E}\in
X_{E}$. We define
\[
\left\vert \lbrack\alpha]_{E}\right\vert :=\inf\{L(\kappa):[\alpha
]_{E}=[\kappa]_{E}\}
\]
and
\[
d([\alpha]_{E},[\beta]_{E})=\left\vert \left[  \overline{\alpha}\ast
\beta\right]  _{E}\right\vert \text{.}%
\]

\end{definition}

The proof that $d$ is a metric on $X_{E}$ is essentially identical to the
proof of this statement (Proposition 13) in \cite{PW1}. Likewise, one may
prove that the deck group $\pi_{E}(X)$ acts as isometries. However in
\cite{PW1}, Proposition 14, we also proved that $\phi_{\varepsilon}$ preserves
all distances less than $\varepsilon$. We will verify here an important analog
of this statement, namely:

\begin{lemma}
\label{metric}If $X$ is a metric space, $E$ is an entourage, and
$([\alpha]_{E},[\beta]_{E})\in E^{\ast}$ then $d([\alpha]_{E},[\beta
]_{E})=d(\phi_{E}([\alpha]_{E}),\phi_{E}([\beta]_{E}))$. In particular, if
$F^{2}\subset E$ then the restriction of $\phi_{E}$ to any $F^{\ast}$-ball is
an isometry onto an $F$-ball.
\end{lemma}

\begin{proof}
If $([\alpha]_{E},[\beta]_{E})\in E^{\ast}$ then by definition of $E^{\ast}$,
$[\overline{\alpha}\ast\beta]_{E}=[x_{n},y_{m}]_{E}$, where $x_{n},y_{m}$ are
the endpoints of $\alpha,\beta$, respectively. Since $[x_{n},y_{m}]_{E}$ is
the shortest possible $E$-chain joining $x_{n},y_{m}$, $d([\alpha]_{E}%
,[\beta]_{E})=d(x_{n},y_{m})=d(\phi_{E}([\alpha]_{E}),\phi_{E}([\beta]_{E}))$.

To prove the second part, suppose that $[\alpha]_{E},[\beta]_{E}\in
B([\gamma]_{E},F^{\ast})$. By definition, $\left(  [\alpha]_{E},[\gamma
]_{E}\right)  \in F^{\ast}$ and $\left(  [\beta]_{E},[\gamma]_{E}\right)  \in
F^{\ast}$. This in turn means that if the endpoints of $\alpha,\beta,\gamma$
are $x_{n},y_{m},z_{k}$, respectively, then $\left[  \overline{\alpha}%
\ast\gamma\right]  _{E}=[x_{n},z_{k}]_{E}$ and $\left[  \overline{\gamma}%
\ast\beta\right]  _{E}=[z_{k},y_{m}]_{E}$ with $(x_{n},z_{k}),(z_{k},y_{m})\in
F$. This implies that $(x_{n},y_{m})\in F^{2}\subset E$. Next note that
\[
\left[  \overline{\alpha}\ast\beta\right]  _{E}=\left[  \overline{\alpha}%
\ast\gamma\ast\overline{\gamma}\ast\beta\right]  _{E}=[x_{n},z_{k},z_{k}%
,y_{m}]_{E}=[x_{n},z_{k},y_{m}]_{E}\text{.}%
\]
Since $(x_{n},y_{m})\in E$, removing $z_{k}$ is a legal move for an
$E$-homotopy, proving that $([\alpha]_{E},[\beta]_{E})\in E^{\ast}$. By the
first part of this lemma, $d([\alpha]_{E},[\beta]_{E})$ is preserved by
$\phi_{E}$. We know from \cite{PW1} (and it is easy to verify) that the
restriction of $\phi_{E}$ to any $F^{\ast}$-ball is a bijection onto an
$F$-ball, finishing the proof.
\end{proof}

Since $\phi_{E}$ is a local isometry, it preserves the lengths of curves. If
$X$ is a geodesic space then the lifted metric is a geodesic metric and
$\phi_{E}$ is distance non-increasing. In fact, one can check that in this
case the metric we have defined is the unique metric with these properties
(cf. \cite{PW1}, Proposition 23).

\section{Refinement and Approximation}

The underlying assumption for the main results in \cite{PW1} and \cite{PW2} is
that the space in question is a geodesic space. Two fundamental issues appear
when attempting to extend results from metric entourages in geodesic spaces to
entourages in general. First, there is the issue of refinement. In a geodesic
space, when $0<\delta<\varepsilon$, any $\varepsilon$-chain $\alpha
=\{x_{0},...,x_{n}\}$ can always be \textquotedblleft
refined\textquotedblright\ into a $\delta$-chain that is in the same
$\varepsilon$-homotopy class as $\alpha$. This is accomplished simply by
subdividing a geodesic joining $x_{i}$ and $x_{i+1}$. Lack of some method of
refinement in some sense \textquotedblleft causes\textquotedblright\ the
problems observed in \cite{Cetal} and \cite{WD} concerning the HCS in
non-geodesic metric spaces. As we will see, restricting to chained entourages
solves this problem.

The second issue in this generality is that metric entourages are totally
ordered by inclusion, but entourages in general are not. This issue is
unavoidable and has many implications. For example, as soon as there are two
entourage covers of a space, neither of which covers the other, then it is
impossible that those two entourage covers may be simultaneously an
$\varepsilon_{1}$-cover and an $\varepsilon_{2}$-cover for a single geodesic
metric. Hence it is generally impossible to realize all entourage covers as
$\varepsilon$-covers of a single geodesic metric. However, recall that Theorem
\ref{TD3} shows that for smooth manifolds of dimension $\geq3$, every
entourage cover is realized as an $\varepsilon$-cover for \textit{some}
Riemannian metric.

\begin{definition}
\label{refdef}Let $X$ be a uniform space, $E\subset F$ be entourages, and
$\alpha=\{x_{0},...,x_{n}\}$ be an $F$-chain. An $E$-refinement of $\alpha$ is
an $E$-chain of the form
\[
\{x_{0}=m_{00},...,m_{0k_{0}}=x_{1},...,x_{r}=m_{r0},...,m_{rk_{r}}%
=x_{r+1},...,x_{n}\}
\]
such that each $E$-chain $\{x_{j}=m_{j0},...,m_{jk_{j}}=x_{j+1}\}$ lies
entirely in $B(x_{j},F)\cap B(x_{j+1},F)$.
\end{definition}

\begin{remark}
\label{refinerem}As mentioned above, in \cite{PW1} we defined
\textquotedblleft$\varepsilon$-refinement\textquotedblright\ of a $\delta
$-chain $\alpha$ in a geodesic space using subdivisions of geodesics joining
sucessive pairs of points in $\alpha$. An $\varepsilon$-refinement as defined
in that paper is an $E_{\varepsilon}$-refinement in the sense of the present
paper, but due to the special method of construction, not every
$E_{\varepsilon}$-refinement in the present sense (even in a geodesic space)
is an $\varepsilon$-refinement in the sense of \cite{PW1}. For this reason, we
will maintain a distinction between $\varepsilon$-refinements and
$E_{\varepsilon}$-refinements. Note also that we cannot expect existence of
$\varepsilon$-refinements even in a geodesic space if the refinement involves
a non-metric entourage $E$. This is because it is possible that no geodesic
joining a pair $(x,y)\in E$ stays inside $B(x,E)\cap B(y,E)$.
\end{remark}

\begin{remark}
\label{confusion}There is some potential for confusion because any $F$-chain
$\alpha$ may be considered as a $D$-chain when $F\subset D$, and the notion of
$E$-refinement depends on whether we consider $\alpha$ as an $F$-chain or a
$D$-chain. That is, an $E$-refinement when $\alpha$ is considered as an
$F$-chain is always an $E$-refinement when $\alpha$ is considered as a $D
$-chain, but not always conversely.
\end{remark}

\begin{remark}
Note that it is immediate from the defintion in the Introduction that if $E$
is a chained entourage and $F\subset E$ then every $E$-chain has an
$F$-refinement. We will see later (Lemmas \ref{refineit} and \ref{reflem})
that $F$-refinements do not change the $E$-homotopy class and have a size that
can be uniformly controlled in any Gromov-Hausdorff precompact class.
\end{remark}

Recall that a subset $A$ of a uniform space $X$ is called \textit{chain
connected} if for every pair of points $x,y\in A$ and entourage $E$ there is
an $E$-chain in $A$ joining $x$ and $y$ (see \cite{BPUU} and note that this
definition is equivalent to what is sometimes known as \textit{uniformly
connected} in the literature). It is easy to check that connected implies
chain connected but the converse is not true (e.g. the rational numbers). For
compact subsets of uniform spaces, connected and chain connected are
equivalent. It is also easy to check that if $X$ has a basis consisting of
entourages with open, connected balls then $X$ is connected if and only if $X
$ is chain connected. In particular, any chain connected geodesic space is connected.

\begin{remark}
We are using chain connectedness rather than connectedness not simply to gain
a little extra generality. Chain connectedness is a far more natural condition
in the context of these discrete methods; many arguments are simplified using
it even when the sets in question are ultimately known to be connected; see
for example Lemmas \ref{conna} and \ref{refineit}. Two easy-to-check and
useful properties of chain connected sets that are not true for connected sets
are: the closure $\overline{A}$ of a set $A$ is chain connected if and only if
$A$ is chain connected; and the Hausdorff limit of a sequence of chain
connected subsets of a metric space is chain connected.
\end{remark}

The next lemma complements the definition in the Introduction, and we will use
it without reference.

\begin{lemma}
\label{shortcut}An entourage $E$ in a uniform space is chained if and only if
$E$ is contained in the closure of its interior and for every $(x,y)\in E$
there is a chain connected set $C\subset B(x,E)\cap B(y,E)$ that contains $x$
and $y$.
\end{lemma}

\begin{proof}
Sufficiency is obvious. Suppose that for every $(x,y)\in E$ and entourage $F$
there is an $F$-chain joining $x$ and $y$ in $B(x,E)\cap B(y,E)$. Let $C$
denote the set of all points $z$ in $B(x,E)\cap B(y,E)$ such that for every
entourage $F$, $z$ is joined to $x$ and to $y$ by an $F$-chain in $B(x,E)\cap
B(y,E)$. If $v,w\in C$ then for any $F$, we may join $v$ to $x$ with an
$F$-chain, then $x$ to $y$ with an $F$-chain, then $y$ to $w$ with an
$F$-chain, all of them lying in $B(x,E)\cap B(y,E)$. In other words, $C$ is a
chain connected subset of $B(x,E)\cap B(y,E)$, and since it also contains
$x,y$, the proof is complete.
\end{proof}

\begin{remark}
Note that the above lemma implies that every $B(x,E)$ is chain connected. This
in turn has another important consequence: since $E$-balls are chain
connected, $X_{E}$ is chain connected and for any $F\subset E$, $\phi
_{EF}:X_{F}\rightarrow X_{E}$ is surjective (Proposition 71 in \cite{BPUU}).
\end{remark}

We will use the following lemma only in the case of a union of two entourages
but this statement isn't much harder to prove:

\begin{lemma}
\label{union}If $\{E_{\alpha}\}$ is a collection of entourages in a uniform
space $X$ then $E:=%
{\displaystyle\bigcup\limits_{\alpha}}
E_{\alpha}$ is an entourage and for any $x\in X$, $B(x,E)=%
{\displaystyle\bigcup\limits_{\alpha}}
B(x,E_{\alpha})$. In addition, if each $E_{\alpha}$ is chained then $E$ is chained.
\end{lemma}

\begin{proof}
Clearly $E$ is an entourage. Now $y\in B(x,E)$ if and only if $(x,y)\in
E_{\alpha}$ for some $\alpha$, which is equivalent to $y\in B(x,E_{\alpha})$
for some $\alpha$. This proves $B(x,E)=%
{\displaystyle\bigcup\limits_{\alpha}}
B(x,E_{\alpha})$. Next observe that $B(x,E)\cap B(y,E)$ contains
$B(x,E_{\beta})\cap B(y,E_{\beta})$ for any $\beta$, which itself contains a
chain connected set containing $x$ and $y$ because $E_{\beta}$ is chain
connected. It is an exercise in elementary point set topology to show that $E
$ is contained in the closure of its interior.
\end{proof}

In a geodesic space, not only does every $\delta$-chain $\alpha$ have an
$\varepsilon$-refinement when $\delta<\varepsilon$, we can control its size in
terms of $\nu(\alpha)$. This is because we may choose points of distance
arbitrarily close to $\delta$ along a geodesic joining $x_{i},x_{i+1}$, which
has length less than $\varepsilon$. So we may always refine by adding at most
$\frac{\delta}{\varepsilon}$ points between any point and its successor. For
more general chained entourages, we are only able to control this number in
compact spaces (but uniformly in any Gromov-Hausdorff precompact class).

\begin{lemma}
\label{conna}Let $X$ be a compact metric space and $\varepsilon>0$. Then any
two points in a chain connected subset $A$ of $X$ may be joined by an
$\varepsilon$-chain in $A$ having at most $2C(X,\frac{\varepsilon}{2})$ points.
\end{lemma}

\begin{proof}
By definition, for any $x,y\in A$ there is some $\varepsilon$-chain
$\{x=x_{0},...,x_{n}=y\}$ with a minimal number of points in $A$ joining
$x,y$. Let $Z$ be an $\frac{\varepsilon}{2}$-dense in $X$ having
$C(X,\frac{\varepsilon}{2})$ points. For each $x_{i}$, choose some $z_{i}$
such that $d(x_{i},z_{i})<\frac{\varepsilon}{2}$. We claim that each element
of $Z$ can be paired in this way at most twice, which completes the proof. To
prove the claim, suppose that $z_{i}=z_{j}$ for some $i<j$. Then by the
triangle inequality $d(x_{i},x_{j})<\varepsilon$. If $j>i+1$ then we could
reduce the size of the chain by eliminating the points $x_{i+1},...,x_{j-1}$,
contradicting minimality. In other words, if a point of $Z$ is used more than
once, it must be used for precisely two adjacent points in the chain.
\end{proof}

\begin{lemma}
\label{refineit}Suppose that $E$ is a chained entourage in a compact metric
space $X$, and $E_{\varepsilon}\subset F\subset E$ for some entourage $F$ and
$\varepsilon>0$. Then every $E$-chain $\alpha$ has an $F$-refinement
$\alpha^{\prime}$ such that $\nu(\alpha^{\prime})\leq2\nu(\alpha)\cdot
C(X,\frac{\varepsilon}{2})$.
\end{lemma}

\begin{proof}
It suffices to show that an $E$-chain $\{x_{0},x_{1}\}$ has an $E_{\varepsilon
}$-refinement of at most $2C(X,\frac{\varepsilon}{2})$ points; such
refinements are also $F$-refinements and may then be concatenated for a longer
chain. By definition of chained entourage, $x_{0}$ and $x_{1}$ lie in a chain
connected set $A$ contained in $B(x_{0},E)\cap B(x_{1},E)$, and we may apply
Lemma \ref{conna} to find an $\varepsilon$-chain $\{x_{0}=z_{0},...,z_{n}%
=x_{1}\}$ with $n\leq2C(X,\frac{\varepsilon}{2})$.
\end{proof}

The next lemma, while simple, is important because it says that all
refinements stay in the same $E$-homotopy class.

\begin{lemma}
\label{reflem}Let $X$ be a uniform space and $F\subset E$ be entourages. If
$\alpha=\{x_{0},...,x_{n}\}$ is an $F$-chain in $B(x_{0},E)$ then $\alpha$ is
$E$-homotopic to $\{x_{0},x_{n}\}$. In particular, every $F$-refinement of an
$E$-chain $\beta$ is $E$-homotopic to $\beta$ and any two $F$-refinements of
$\beta$ are $E$-homotopic.
\end{lemma}

\begin{proof}
Since $x_{2}\in B(x_{0},E)$, $\{x_{0},\overbrace{x_{1}},x_{2}...x_{n}\}$ is a
legal move. Likewise $\{x_{0},\overbrace{x_{2}},x_{3},...x_{n}\}$ is a legal
move and the proof of the first statement is finished after finitely many
steps. The last statements are immediate consequences of the first and the definitions.
\end{proof}

The next lemma formalizes a process that is a discrete version of the
Arzela-Ascoli Theorem, used several times in this paper and originating in
\cite{PW1}.

\begin{lemma}
[Chain Normalizing]\label{norm}Let $X$ be a compact metric space, $E$ be an
entourage in $X$, and $\alpha_{i}=\{x_{i0},x_{i1},...,x_{in_{i}}\}$ be a
sequence of $E$-chains of bounded length. Then:

\begin{enumerate}
\item Up to $E$-homotopy we may assume that for some $n$, $n_{i}=n$ for all
$i$.

\item By choosing a subsequence if necessary we may assume that for all $0\leq
j\leq n$, $x_{ij}\rightarrow x_{j}$ for some $x_{j}\in X$.
\end{enumerate}
\end{lemma}

\begin{proof}
Let $L$ be an upper bound for $\{L(\alpha_{i})\}$ and suppose that
$E_{\varepsilon}\subset E$. For any $i$, choose some representative of
$[\alpha_{i}]_{E}$ of length $L_{i}\leq L$. We may assume that the maximum
number of values $d(x_{ij},x_{i(j+1)})$ that are smaller than $\frac
{\varepsilon}{2}$ is at most $\frac{n_{i}}{2}$. Otherwise there would have to
be two consecutive distances smaller than $\frac{\varepsilon}{2}$ and we could
remove one point while staying in the same $E$-homotopy class and not
increasing length. In other words, $L\geq L_{i}\geq\frac{n_{i}}{2}\cdot
\frac{\varepsilon}{2}$ and we conclude that $n_{i}\leq\frac{4L}{\varepsilon}$.
Now by adding repeated points, if needed, we may ensure that $\nu(\alpha_{i})$
is precisely $n:=\left\lfloor \frac{4L}{\varepsilon}\right\rfloor +1$. The
second part is an immediate consequence of compactness.
\end{proof}

For many problems it is important to have some version of \textquotedblleft
close $E$-chains are $E$-homotopic\textquotedblright. Proposition \ref{close}
is analogous to Proposition 15 in \cite{PW1}, using $E=E_{\varepsilon}$ and
$F\subset E_{\frac{\varepsilon}{2}}$. The proof here uses the same homotopy
steps as the one in \cite{PW1}, but due to the differences in assumptions we
write out a full proof here.

\begin{proposition}
\label{close}Let $X$ be a uniform space, $E$ be an entourage, and $F$ be an
entourage such that $F^{2}\subset E$. If $\alpha=\{x_{0},...,x_{n}\}$ is an
$F$-chain and $\beta=\{x_{0}=y_{0},...,y_{n}=x_{n}\}$ is a chain such
$(x_{i},y_{i})\in F$ for all $i$, then $\beta$ is an $E$-chain that is
$E$-homotopic to $\alpha$.
\end{proposition}

\begin{proof}
We will construct an $E$-homotopy $\eta$ from $\alpha$ to $\beta$, using the
fact that $F^{2}\subset E$ to see that in each step that the resulting chain
is an $E$-chain, i.e. the step is legal. For example, the first step in the
second line below is justified by the fact that $(x_{0},x_{1})\in F$ and
$(x_{1},y_{1})\in F$, so $(x_{0},y_{1})\in F^{2}\subset E$. The remaining
justifications are similar.
\[
\alpha=\{x_{0},x_{1},...,x_{n}\}\rightarrow\{x_{0},\overbrace{x_{1}}%
,x_{1},...,x_{n}\}\rightarrow\{x_{0},x_{1},\overbrace{y_{1}},x_{1},...,x_{n}\}
\]%
\[
\rightarrow\{x_{0},\underbrace{x_{1}},y_{1},x_{1},...,x_{n}\}\rightarrow
\{x_{0},y_{1},\underbrace{x_{1}},x_{2},...,x_{n}\}
\]%
\[
\rightarrow\{x_{0},y_{1},\overbrace{x_{2}},x_{2},...,x_{n}\}\rightarrow
\{x_{0},y_{1},x_{2},\overbrace{y_{2}},x_{2},...,x_{n}\}
\]%
\[
\rightarrow\{x_{0},y_{1},\underbrace{x_{2}},y_{2},x_{2},...,x_{n}%
\}\rightarrow\{x_{0},y_{1},y_{2},\underbrace{x_{2}},x_{3},...,x_{n}%
\}\rightarrow\cdot\cdot\cdot\rightarrow\beta
\]

\end{proof}

\begin{proposition}
\label{close2}Let $X$ be a metric space and $E$ be a chained entourage.
Suppose that $\alpha_{i}:=\{x_{0}=x_{i0},...,x_{in}=x_{n}\}$ is a sequence of
$E$-chains converging to $\alpha=\{x_{0},...,x_{n}\}$. Then $\alpha$ is an
$\overline{E}$-chain such that for all sufficiently large $i$, $\alpha_{i}$ is
$E$-homotopic to some (hence any) $E$-refinement of $\alpha$.
\end{proposition}

\begin{proof}
The Ball Continuity Lemma (Lemma \ref{bc}) implies that $\alpha$ is an
$\overline{E}$-chain. Note that in any metric space and for any $0<\varepsilon
<\delta$, by continuity of the distance function,
\[
\overline{E_{\varepsilon}}\subset\{(x,y):d(x,y)\leq\varepsilon\}\subset
E_{\delta}\text{.}%
\]
In particular we may find some \textit{closed} entourage $F$ such that
$F^{2}\subset E$. Let $\alpha_{i}^{\prime}$ be an $F$-refinement of $\alpha$
for every $i$. Since $\nu(\alpha_{i})=n$ for all $i$, $L(\alpha_{i}%
)\rightarrow L(\alpha)$ and hence $\{L(\alpha_{i})\}$ is bounded. By the Chain
Normalizing Lemma (Lemma \ref{norm}) we may suppose that for some fixed $k$,
and for all $i$,
\[
\alpha_{i}^{\prime}=\{x_{i0}=m_{00}^{i},...,m_{0k}^{i}=x_{i1},...,x_{ir}%
=m_{r0}^{i},...,m_{rk}^{i}=x_{r+1},...,x_{n}\}
\]
and that for all $j,m$, $m_{mj}^{i}\rightarrow m_{mj}$, for some $m_{j}\in X$.
In addition,
\[
\alpha^{\prime}:=\{x_{0}=m_{00},...,m_{0k}=x_{1},...,x_{r}=m_{r0}%
,...,m_{rk}=x_{r+1},...,x_{n}\}
\]
is an $F$-chain since $F$ is closed. By the Ball Continuity Lemma (Lemma
\ref{bc}), each $m_{jk}$ lies in $B(x_{j},\overline{E})\cap B(x_{j+1}%
,\overline{E})$ and therefore $\alpha^{\prime}$ is an $F$-refinement, hence an
$E$-refinement, of the $\overline{E}$-chain $\alpha$. For all large enough $i$
and all $j,m$, $(m_{mj}^{i},m_{mj})\in F$, and Proposition \ref{close} tells
us that $\alpha_{i}^{\prime}$ is $E$-homotopic to $\alpha^{\prime}$. Since
each $\alpha_{i}$ is $E$-homotopic to $\alpha_{i}^{\prime}$ by Lemma
\ref{reflem}, the proof is complete.
\end{proof}

\begin{corollary}
\label{closefree}Let $X$ be a metric space and $E$ be a chained entourage.
Suppose that $\alpha_{i}:=\{x_{i0},...,x_{in}=x_{i0}\}$ is a sequence of
$E$-loops converging to an $\overline{E}$-loop $\alpha=\{x_{0},...,x_{n}%
=x_{0}\}$. Then for all sufficiently large $i$, $\alpha_{i}$ is freely
$E$-homotopic to some (hence any) $E$-refinement of $\alpha$.
\end{corollary}

\begin{proof}
Let $\beta=\{\ast=y_{0},...,y_{m}=x_{0}\}$ be any $E$-chain. Then for large
enough $i$, $(x_{0},x_{i0})\in E$ and therefore
\[
\alpha_{i}^{\prime}:=\{\ast=y_{0},...,x_{0},x_{i0}\}\ast\alpha_{i}\ast
\{x_{i0},x_{0},...,y_{0}=\ast\}
\]
is an $E$-loop at $\ast$. Then $\alpha_{i}^{\prime}\rightarrow\alpha^{\prime
}:=\overline{\beta}\ast\alpha\ast\beta$ and by Proposition \ref{close2}, for
all large $i$, $\alpha_{i}^{\prime}$ is $E$-homotopic to any $E$-refinement of
the $\overline{E}$-loop $\alpha^{\prime}$. By definition of freely
$E$-homotopic, the proof is finished.
\end{proof}

\begin{remark}
\label{proper}While we will not need it here, in analogy with the notion of
\textquotedblleft proper metric space\textquotedblright\ (meaning all closed
metric balls are compact), one could define a \textquotedblleft proper uniform
space\textquotedblright\ to be a uniform space $Y$ such that $Y\times Y$ has a
basis of, and is the countable union of compact entourges. Of course covering
spaces of a compact space $X$ may not be compact, but one can show that
$X_{E}$ is still a proper uniform space for any chained entourage $E$.
Proposition \ref{close2} and various other results in this paper may be
applied to proper uniform spaces.
\end{remark}

\section{Curves and the Fundamental Group}

Many of the next few results extend basic theorems concerning covering space
theory both discrete and classical, with $\phi_{E}$ playing the role of the
universal covering map. The standard proofs of some of these results in the
classical theory use the homotopy lifting property, which one could try to
emulate here: i.e. lifting $F$-homotopies to $F^{\ast}$-homotopies in $X_{E}
$. However, it is not clear whether this approach is worth the effort. For the
results we need, we are able to get by only with chain lifting. See also
\cite{W} for some lifting results concerning $\varepsilon$-chains in geodesic spaces.

\begin{lemma}
[Chain Lifting]\label{cl}Let $X$ be a uniform space and $E$ be an entourage.
Suppose that $\beta:=\{x_{0},...,x_{n}\}$ is an $E$-chain and $[\alpha]_{E}$
is such that $\phi_{E}([\alpha]_{E})=x_{0}$. Let $y_{i}:=[\alpha\ast
\{x_{0},...,x_{i}\}]_{E}$. Then $\widetilde{\beta}:=\{y_{0}=[\alpha]_{E}%
,y_{1},...,y_{n}=[\alpha\ast\beta]_{E}\}$ is the unique \textquotedblleft
lift\textquotedblright\ of $\beta$ to $[\alpha]_{E}$. That is,
$\widetilde{\beta}$ is the unique $E^{\ast}$-chain in $X_{E}$ starting at
$[\alpha]_{E}$ such that $\phi_{E}(\widetilde{\beta})=\beta$.
\end{lemma}

\begin{proof}
Since the endpoint of $\alpha\ast\{x_{0},...,x_{i}\}$ is $x_{i}$, $\phi
_{E}(y_{i})=x_{i}$; i.e., $\phi_{E}(\widetilde{\beta})=\beta$. Also, by
definition of $E^{\ast}$, $\widetilde{\beta}$ is an $E^{\ast}$-chain (in fact
this is easiest to see using the original definition from \cite{BPUU}, see
Remark \ref{star}). So we need only show uniqueness, which we will prove by
induction on $n$. For $n=0$, the starting point $[\alpha]_{E}$ is determined
by assumption, so $\{[\alpha]_{E}\}$ is the unique lift. Suppose we have
proved the statement for $n-1$. Then given $\beta=\{x_{0},...,x_{n}\}$ and
using the above notation we know that $\{y_{0},...,y_{n-1}\}$ is the unique
lift of $\{x_{0},...,x_{n-1}\}$. But we already know that $\phi_{E}$ is
injective from $B(y_{n-1},E^{\ast})$ onto $B(x_{n-1},E)$ and therefore $x_{n}$
has a unique preimage in $B(y_{n-1},E^{\ast})$. Therefore $y_{n}$, which lies
in $B(y_{n-1},E^{\ast}) $ and satisfies $\phi_{E}(y_{n})=x_{n}$, is the only possibility.
\end{proof}

\begin{lemma}
\label{action}Let $X$ be a uniform space and $E$ be an entourage. Let
$[\lambda]_{E}\in\pi_{E}(X)$ and $[\alpha]_{E}\in X_{E}$. Then considering
$[\lambda]_{E}$ as a deck transformation of $X_{E}$, $[\lambda]_{E}%
([\alpha]_{E})$ is the endpoint of the lift of $\alpha$ starting at
$[\lambda]_{E}$ in $X_{E}$. In particular, the unique lift of $\alpha$
starting at the basepoint ends at $[\alpha]_{E}$.
\end{lemma}

\begin{proof}
By definition, $[\lambda]_{E}([\alpha]_{E})=[\lambda\ast\alpha]_{E}$, which by
the Chain Lifting Lemma (Lemma \ref{cl}) is the endpoint of the lift of
$\lambda\ast\alpha$ to $[\ast]_{E}$. But by uniqueness this is the endpoint of
the lift of $\alpha$ starting at the endpoint $z$ of the lift
$\widetilde{\lambda}$ of $\lambda$ starting at $[\ast]_{E}$. By the Chain
Lifting Lemma, $z=[\lambda]_{E}$, completing the proof.
\end{proof}

We also need to extend the definition of \textquotedblleft
stringing\textquotedblright\ from \cite{PW1}, which among other things allows
us to understand the relationship between the fundamental group and the groups
$\pi_{E}(X)$. Note that the definition below is slightly stronger than the one
in \cite{PW1}, but the new one is more natural and necessary for the current
paper; replacing the definition in \cite{PW1} by the new one would have
essentially no impact on \cite{PW1}.

\begin{definition}
Let $\alpha:=\{x_{0},...,x_{n}\}$ be an $E$-chain in a metric space $X$, where
$E$ is an entourage in a uniform space. A stringing of $\alpha$ consists of a
path $\widehat{\alpha}$ formed by concatenating paths $\gamma_{i}$ from
$x_{i}$ to $x_{i+1}$ where each path $\gamma_{i}$ lies entirely in
$B(x_{i},E)\cap B(x_{i+1},E)$. Conversely, let $c:[0,1]\rightarrow X$ be a
curve. An $E$-subdivision of $c$ is an $E$-chain $\beta=\{c(t_{0}%
),...,c(t_{n})\}$, where $0=t_{0}\leq\cdot\cdot\cdot\leq t_{n}=1$ are such
that $c$ is a stringing of $\beta$.
\end{definition}

\begin{remark}
Note that, unlike the typical definition of partitions, our definition of
$E$-subdivision allows \textquotedblleft repeated points\textquotedblright%
\ (as do chains in general) due to the inequality $t_{i}\leq t_{i+1}$. This
simplifies matters, for example when considering limits of chains and curves.
\end{remark}

\begin{lemma}
\label{subdiv}If $X$ is a metric space, $c:[0,1]\rightarrow X$ is a curve and
$E$ is an entourage in $X$, then $c$ has an $E$-subdivision. Moreover, any two
$E$-subdivisions of $c$ are $E$-homotopic and curves $c_{1}$ and $c_{2}$ with
the same endpoints are $E$-homotopic if and only if any $E$-subdivision of
$c_{1}$ is $E$-homotopic to any $E$-subdivision of $c_{2}$.
\end{lemma}

\begin{proof}
For existence, it suffices to find an $E_{\varepsilon}$-subdivision for some
$E_{\varepsilon}\subset E$. Since $c$ is uniformly continuous there exists
some $\delta>0$ such that if $\left\vert s-t\right\vert <\delta$ then
$d(c(s),c(t))<\varepsilon$. Now subdivide $[0,1]$ into intervals of length
less than $\delta$ with endpoints $t_{0}=0<t_{1}<\cdot\cdot\cdot<t_{n}=1$.
Then the image of the restriction of $c$ to any interval $[t_{i},t_{i+1}]$
lies entirely in $B(c(t_{i}),\varepsilon)\cap B(c(t_{i+1}),\varepsilon)$ and
hence the $\varepsilon$-chain $\{c(t_{0}),...,c(t_{n})\}$ is an
$E_{\varepsilon}$-subdivision.

For the second statement, suppose that $\alpha_{1},\alpha_{2}$ are
$E$-subdivisions of $c$, corresponding to partitions $\tau_{1},\tau_{2}$ of
$[0,1]$. We claim that $s=\tau_{1}\cup\tau_{2}$ has the property that if
$s=\{s_{0}\leq\cdot\cdot\cdot\leq s_{m}\}$ then $\alpha:=\{c(s_{0}%
),...,c(s_{m})\}$ is an $E$-refinement of both $\alpha_{1}$ and $\alpha_{2}$.
This will complete the proof by Lemma \ref{reflem}. Moreover, by symmetry it
suffices to show that $\alpha$ is an $E$-refinement of $\alpha_{1}$. Consider
some $t_{i}\leq t_{i+1}\in\tau_{1}$, and suppose that for some $j\leq k$,
$t_{i}\leq s_{j}\leq\cdot\cdot\cdot\leq s_{k}<t_{i+1}$ with $s_{m}\in\tau_{2}%
$. Since $\alpha_{1}$ is an $E$-subdivision of $c$, the restriction of $c$ to
$[t_{i},t_{i+1}]$ lies entirely in $B(c(t_{i}),E)\cap B(c(t_{i+1}),E)$ and
hence each of the points $c(s_{m})$, which are the added points of $\alpha
_{2}$, lie in $B(c(t_{i}),E)\cap B(c(t_{i+1}),E)$. This shows that $\alpha$ is
an $E$-refinement of $\alpha_{1}$, as required.
\end{proof}

\begin{lemma}
\label{strex}Suppose that $X$ is a Peano continuum, $E$ is a chained entourage
and $\alpha$ is an $E$-chain. Then for some chained entourage $F\subset E$,
$\alpha$ has an $F$-refinement with a stringing. In particular, $[\alpha]_{E}$
contains a representative with a stringing.
\end{lemma}

\begin{proof}
Let $X$ have a geodesic metric (by the Bing-Moise Theorem). Since $E$ is a
chained entourage, there is some $E_{\varepsilon}$-refinement $\beta
=\{x_{0},...,x_{n}\}$ of $\alpha$ when $E_{\varepsilon}\subset E$. Now $x_{i}$
and $x_{i+1}$ may be joined by a geodesic, which remains inside $B(x_{i}%
,E_{\varepsilon})\cap B(x_{i+1},E_{\varepsilon})\subset B(x_{i},E)\cap
B(x_{i+1},E)$. The concatenation of these geodesics is an $F$-stringing, hence
an $E$-stringing of $\beta$.
\end{proof}

\begin{remark}
\label{pnc}Note that the stringing of $\beta$ in the above lemma may not be a
stringing of $\alpha$ since there is no reason why the concatenated curve must
lie enside $B(y_{i},E)\cap B(y_{i+1},E)$, where $y_{i}$ and $y_{i+1}$ are
consecutive points in $\alpha$. Note also that the proof of the lemma only
requires that $X$ be a uniform space with a uniformly equivalent geodesic
metric--for example any smooth, possibly non-compact, manifold with the
uniform structure given by a particular Riemannian metric.
\end{remark}

\begin{definition}
\label{ehomc}If $c_{1},c_{2}$ are curves in a uniform space starting and
ending at the same point,and $E$ is an entourage, we say that $c_{1}$ and
$c_{2}$ are $E$-homotopic (resp. freely $E$-homotopic) if there exist
$E$-subdivisions $\kappa_{i}$ of $c_{i}$, that are $E$-homotopic (resp. freely
$E$-homotopic).
\end{definition}

\begin{lemma}
\label{ehomo}Let $E$ be an entourage in a uniform space $X$ and $c_{1},c_{2}$
be curves starting and ending at the same point. Then $c_{1}$ and $c_{2}$ are
$E$-homotopic if and only if one lift (equivalently any lifts) of $c_{1}$ and
$c_{2}$ to $X_{E}$ starting at the same point end at the same point. In
particular, a curve $c$ lifts as a loop to $X_{E}$ if and only if $c$ is
freely $E$-null.
\end{lemma}

\begin{proof}
By arguments analogous to standard ones for covering space theory involving
concatenations and uniqueness of lifts, the entire statement reduces to
showing the following: If $c$ is a loop at the basepoint then $c$ lifts as a
loop at the basepoint if and only if it is $E$-null. Let $\lambda
:=\{\ast=c(t_{0}),...,c(t_{n})=\ast\}$ be a $E$-subdivision of $c$. By the
Chain Lifting Lemma its unique lift $\widetilde{\lambda}$ starting at the
basepoint ends at $[\lambda]_{E}$, i.e. $\widetilde{\lambda}$ is a loop if and
only if $\lambda$, hence $c$, is $E$-null.
\end{proof}

The next two statements are extensions of \cite{PW1}, Proposition 20 and
Corollary 21, replacing $\varepsilon$ by $E$ or $E^{\ast}$ as appropriate.
Note only that $E$ is not assumed to be chained, so that $X_{E}$ may not be
connected (and we are not guaranteed that stringings exist). But this does not
affect the statement or arguments.

\begin{proposition}
\label{ender}Let $E$ be an entourage in a uniform space $X$ and $\alpha$ be an
$E$-chain starting at the basepoint. Then the unique lift of any stringing
$\widehat{\alpha}$ starting at the basepoint $[\ast]_{E}$ in $X_{E}$ has
$[\alpha]_{E}$ as its endpoint.
\end{proposition}

\begin{proof}
By uniqueness in the Chain Lifting Lemma and the second statement in Lemma
\ref{action}, it suffices to show that if $c:[0,1]\rightarrow X$ is a
stringing of $\alpha:=\{x_{0},...,x_{n}\}$ with $x_{i}=c(t_{i})$, and
$\widetilde{c}$ is the unique lift of $c$ to the basepoint, then
$\{\widetilde{c}(t_{0}),...,\widetilde{c}(t_{n})\}$ is an $E^{\ast}$-chain.
But by definition, for any $i$, the segment of $c$ from $x_{i}$ to $x_{i+1}$
lies entirely in $B(x_{i},E)$, so lifts into $B(\widetilde{c}(t_{i}),E^{\ast
})$, proving that $(\widetilde{c}(t_{i}),\widetilde{c}(t_{i+1}))\in E^{\ast} $.
\end{proof}

The next corollary now follows from the homotopy lifting property for covering
spaces, verifying that homotopies of curves are \textquotedblleft stronger
than\textquotedblright\ $E$-homotopies. Recall that even with a chained
entourage in a geodesic space we are not guaranteed that stringings exist, but
Lemma \ref{strex} tells us we can always find a stringing of an $E $-homotopic
$E$-chain.

\begin{corollary}
\label{homimp}Let $E$ be an entourage in a metric space $X$ and $\alpha,\beta$
be $E$-chains. If there exist stringings $\widehat{\alpha}$ and
$\widehat{\beta}$ that are path homotopic then $\alpha$ and $\beta$ are $E$-homotopic.
\end{corollary}

\begin{corollary}
\label{finally}Let $X$ be a Peano continuum, $\lambda=\{x_{0},...,x_{n}\}$ be
an $E$-loop, and $c_{i}$ be a path from $x_{i}$ to $x_{i+1}$, denoting by $c$
the concatenation of those paths. If the lift of $c$ to any point
$[\alpha]_{E}$ in $X_{E}$ ends at $[\alpha\ast\lambda]_{E}$, then $c$ is
$E$-homotopic to some, hence any stringing of $\lambda$.
\end{corollary}

\begin{proof}
The statement and the proof use the Chain Lifting Lemma (Lemma \ref{cl}), and
we will use it without further explicit reference. First, in order for the
hypothesis to make sense, $\alpha$ must be an $E$-chain from the basepoint to
$x_{0}$. By Lemma \ref{strex}, up to $E$-homotopy we may let $\widehat{\alpha
}$ be a stringing of $\alpha$. Then $\widehat{\lambda}\ast\widehat{\alpha}$ is
a stringing of $\lambda\ast\alpha$. According to Proposition \ref{ender}, the
endpoint of the unique lift $\widetilde{\widehat{\lambda}\ast\widehat{\alpha}%
}$ of $\widehat{\lambda}\ast\widehat{\alpha}$ is $[\lambda\ast\alpha]_{E}$. By
assumption, the lift $\widetilde{c}$ of $c$ to $[\alpha]_{E}$ also ends in
$[\lambda\ast\alpha]_{E}$. By Lemma \ref{ehomo}, $\widehat{\lambda}%
\ast\widehat{\alpha}$ is $E$-homotopic to $c\ast\widehat{\alpha}$ and
therefore $c$ is $E$-homotopic to $\widehat{\lambda}$ as required.
\end{proof}

The above Corollary is useful because it \textquotedblleft frees
us\textquotedblright\ from having to rely on stringings in some situations;
for example the next proposition is useful in the proof of Theorem
\ref{final'}.

\begin{proposition}
\label{finally2}Let $X$ be a geodesic space and $E$ be a chained entourage. Then

\begin{enumerate}
\item If $\lambda$ is an $E$-loop then there is a curve loop $c$ such that
some (hence any) $E$-subdivision of $c$ is $E$-homotopic to $\lambda$ and
$L(c)=L(\lambda)$.

\item Conversely, if $c$ is a curve loop then there is an $E$-loop $\lambda$
such that $L(c)=L(\lambda)$ and $\lambda$ is $E$-homotopic to some (hence any)
$E$-subdivision of $c$.
\end{enumerate}
\end{proposition}

\begin{proof}
For the first part, let $X_{E}$ have the lifted geodesic metric (see the end
of the second section). Let $\widetilde{\lambda}$ be the unique lift of
$\lambda$ at a point $[\alpha]_{E}\in X_{E}$, which is an $E^{\ast}$-chain. By
Lemma \ref{metric}, $L(\widetilde{\lambda})=L(\lambda)$. Since $X_{E}$ is a
geodesic space we may join each point in $\widetilde{\lambda}$ to its
successor by a geodesic, concatenating to produce a curve $\widetilde{c}$ such
that $L(\widetilde{c})=L(\widetilde{\lambda})=L(\lambda)$. Since $\phi_{E}$
preserves the lengths of curves, $c:=\phi_{E}(\widetilde{c})$ has the same
length as $\lambda$. Moreover, the endpoint of $\widetilde{c}$, which by
uniqueness is the lift of $c$ at $[\alpha]_{E}$, is by construction the same
as the endpoint of $\widetilde{\lambda}$. That endpoint is equal to
$[\alpha\ast\lambda]_{E}$ by Proposition \ref{ender}.

Now let $\lambda^{\prime}$ be any $E$-subdivision of $c$; that is, $c$ is a
stringing of $\lambda^{\prime}$. Again by uniqueness of lifts, the lift
$\widetilde{\lambda^{\prime}}$ of $\lambda^{\prime}$ at $[\alpha]_{E}$ lies on
$c$ and also ends at $[\alpha\ast\lambda]_{E}$. By Proposition \ref{ender},
$[\alpha\ast\lambda^{\prime}]_{E}=[\alpha\ast\lambda]_{E} $, and hence
$[\lambda^{\prime}]_{E}=[\lambda]_{E}$, completing the proof of the first part.

For the second part, let $F$ be an open entourage contained in $E$ begin with
any $F$-subdivision $\omega$ of $c$. By definition of the length of a curve
$L(\omega)\leq L(c)$. If the lengths are equal, we are finished. Otherwise,
suppose that $\delta:=L(c)-L(\omega)>0$ and choose any point $x_{i}$ and its
successor $x_{i+1}$ on $\omega$. Since $F$ is open we may choose $n$ large
enough that, setting $\varepsilon:=\frac{\delta}{2n}$, $B(x_{i},F)\cap
B(x_{i+1},F)$ contains $B(x_{i},2\varepsilon)$. Because the space is geodesic,
we may pick a point $z\in B(x_{i},2\varepsilon)$ so that $d(x_{i}%
,z)=\frac{\delta}{2n}$ (for example just apply the Intermediate Value Theorem
to any geodesic from $x_{i}$ to $x_{i+1}$). Now we make the following legal
moves for an $E$-homotopy:%
\[
\{x_{i},\overbrace{z},x_{i+1}\rightarrow\{x_{i},z,\overbrace{x_{i}}%
,x_{i+1}\rightarrow\{x_{i},z,x_{i},\overbrace{z},x_{i+1}%
\]%
\[
\rightarrow\{x_{i},z,x_{i},z,\overbrace{x_{i}},x_{i+1}\rightarrow\cdot
\cdot\cdot
\]
Each pair of moves, adding $z$ then $x_{i}$, increases the length of $\lambda$
by precisely $2\delta$. Therefore after $n$ such moves we have attained
precisely $L(c)$.
\end{proof}

Note that the next statement is not true when replacing \textquotedblleft%
$E$-homotopic\textquotedblright\ with \textquotedblleft
homotopic\textquotedblright, for example if the space is not semi-locally
simply connected.

\begin{proposition}
\label{uniform}Let $c_{i}:[0,1]\rightarrow X$ be a sequence of curves in a
compact metric space $X$ of uniformly bounded length, uniformly convergent to
$c:[0,1]\rightarrow X$, and let $E$ be any chained entourage.

\begin{enumerate}
\item If the $c_{i}$ have the same start and endpoints then $c_{i}$ is
$E$-homotopic to $c$ for all large $i$.

\item If the $c_{i}$ are closed then $c_{i}$ is freely $E$-homotopic to $c$
for all large $i$.
\end{enumerate}
\end{proposition}

\begin{proof}
If the statement were not true then we could, by taking a subsequence, assume
that for all $i$, $c_{i}$ is never $E$-homomotopic to $c$. Let $\varepsilon>0$
be such that $\overline{E_{\varepsilon}}\subset E$. For every $i$, let
$\lambda_{i}=\{x_{0i}=c_{i}(0),x_{1i}=c_{i}(t_{1i}),...,x_{n_{i}i}=c_{i}(1)\}$
be an $E_{\varepsilon}$-subdivision of $c_{i}$, where $0\leq t_{1i}\leq
\cdot\cdot\cdot\leq t_{\left(  n_{i}-1\right)  i}\leq1$. By Lemma
\ref{refineit}, we may assume that these chains have bounded length, and hence
we may apply the Chain Normalizing Lemma (Lemma \ref{norm}). That is, we may
assume that $ni=n$ for some fixed $n$ and all $i$, and $x_{ji}\rightarrow
x_{j}$ for some $x_{j}$, for all $j$. Since $c_{i}$ converges uniformly to
$c$, $x_{j}=c(t_{j})$ for some $t_{j}\in\lbrack0,1]$ with $0\leq t_{1}%
\leq\cdot\cdot\cdot\leq t_{n}=1$. We will next show that $\lambda
=\{x_{0},...,x_{n}\}$ is an $\overline{E_{\varepsilon}}$-subdivision, hence an
$E$-subdivision, of $c$. Let $c_{ji}$ denote the restriction of $c_{i}$ to
$[t_{ji},t_{(j+1)i}]$; by definition of subdivision, the image of $c_{ji}$
lies in $B(x_{ji},E)\cap B(x_{(j+1)i},E)$. Since $\{c_{i}\}$ converges
uniformly, the Ball Continuity Lemma (Lemma \ref{bc}) implies that if $c_{j}$
denotes the restriction of $c$ to $[t_{j},t_{j+1}]$ then $c_{j}$ lies in
$B(x_{j},\overline{E_{\varepsilon}})\cap B(x_{j+1},\overline{E_{\varepsilon}%
})$. This shows that $\lambda$ is an $\overline{E_{\varepsilon}}$-subdivision
of $c$. It now follows from Proposition \ref{close2} that $\lambda$ is an
$\overline{E_{\varepsilon}}$-loop, hence an $E$-loop that is $\varepsilon
$-homotopic (hence $E$-homotopic) to $\lambda_{i}$ for all large $i$. That is,
$c_{i}$ is $E$-homotopic to $c$, completing the proof of the first part.

The proof of the second part is the same as the proof of the first, using
Corollary \ref{closefree} rather than Proposition \ref{close2}.
\end{proof}

\begin{remark}
We do not need it for this paper, but the above statement may be proved with a
little more work without the condition that the curves have uniformly bounded
length. The idea is that even if the curves do not have uniformly bounded
length, the fact that the convergence is uniform allows one to divide any
$c_{i}$ into a concatenation of segments, the number of which is uniformly
bounded and each of which lies in an $\frac{\varepsilon}{2}$-ball. Each of
those segments then has an $E_{\varepsilon}$-subdivision consisting only of
two points, and those together provide an $E_{\varepsilon}$-subdivision of
$c_{i}$ with length bounded independent of $i$.
\end{remark}

The next definition extends a notion from \cite{PW1}:

\begin{definition}
\label{triaddef}If $X$ is a uniform space and $E$ is an entourage, an $E$-loop
of the form $\lambda=\alpha\ast\tau\ast\overline{\alpha}$, where $\nu(\tau
)=3$, will be called $E$-small. Note that this notation includes the case when
$\alpha$ consists of a single point--i.e. $\lambda=\tau$. In this case, we
will call $\lambda$ an $E$-triad.
\end{definition}

Note that any $E$-small loop is $E$-null since two of the points in $\tau$ may
be removed one by one, followed by the points in $\alpha$ and $\overline
{\alpha}$.

The next essential proposition is an extension of Proposition 29 in
\cite{PW1}, with a similar proof (essencially replacing $\varepsilon$ by $E$
and $\delta$ by $D$). However, we write out the complete proof here so it is
clear how our new definition of refinement is used (including the obvious fact
that the reversal of a refinement is a refinement).

\begin{proposition}
\label{deltaimp}Let $X$ be a uniform space, $D$ be a chained entourage in $X$
and $E\subset D$ be an entourage. Suppose $\alpha,\beta$ are $E$-chains and
$\left\langle \gamma_{0},...,\gamma_{n}\right\rangle $ is a $D$-homotopy such
that $\gamma_{0}=\alpha$ and $\gamma_{n}=\beta$. Then $[\beta]_{E}%
=[\lambda_{1}\ast\cdot\cdot\cdot\ast\lambda_{r}\ast\alpha\ast\lambda_{r+1}%
\ast\cdot\cdot\cdot\ast\lambda_{n}]_{E}$, where each $\lambda_{i}$ is an
$E$-refinement of a $D$-small loop.
\end{proposition}

\begin{proof}
We will prove by induction that for every $k\leq n$, an $E$-refinement
$\gamma_{k}^{\prime}$ of $\gamma_{k}$ is $E$-homotopic to $\lambda_{1}%
\ast\cdot\cdot\cdot\ast\alpha\ast\cdot\cdot\cdot\ast\lambda_{k}$, where each
$\lambda_{i}$ is an $E$-refinement of a $D$-small loop. Since any two
$D$-refinements of an $E$-chain are $E$-homotopic (Lemma \ref{reflem}), this
will complete the proof.

The case $k=0$ is trivial. Suppose the statement is true for some $0\leq k<n$.
The points required to $E$-refine $\gamma_{k}$ to $\gamma_{k}^{\prime}\ $will
be denoted by $m_{i}$. Suppose that $\gamma_{k+1}$ is obtained from
$\gamma_{k}$ by adding a point $x$ between $x_{i}$ and $x_{i+1}$. Let
$\{x_{i},a_{1},...,a_{k},x\}$ be an $E$-refinement of $\{x_{i},x\}$ and
$\{x,b_{1},...,b_{m},x_{i+1}\}$ be an $E$-refinement of $\{x,x_{i+1}\}$, so
\[
\gamma_{k+1}^{\prime}:=\{x_{0},m_{0},...,x_{i},a_{1},...,a_{k},x,b_{1}%
,...,b_{m},x_{i+1},m_{r},...,x_{j}\}
\]
is an $E$-refinement of $\gamma_{k+1}$. Defining $\mu_{k+1}:=\{x_{0}%
,m_{0},...,x_{i}\}$ and
\[
\kappa_{k+1}=\{x_{i},a_{1},...,a_{k},x,b_{1},...,b_{m},x_{i+1},m_{r}%
,...,x_{i}\}
\]
we have
\[
\left[  \gamma_{k+1}^{\prime}\right]  _{E}=\left[  \mu_{k+1}\ast\kappa
_{k+1}\ast\overline{\mu_{k+1}}\ast\gamma_{k}^{\prime}\right]  _{E}%
\]
and since the homotopy is a $D$-homotopy, $\lambda_{k+1}:=\mu_{k+1}\ast
\kappa_{k+1}\ast\overline{\mu_{k+1}}$ is an $E$-refinement of a $D$-small
loop. The case when a point is removed from $\gamma_{k}$ is similar, except
that the refined $D$-small loop is multiplied on the right.
\end{proof}

\begin{corollary}
\label{gens}Let $X$ be a uniform space, $D$ be a chained entourage in $X$ and
$E\subset D$ be an entourage. Then $\ker\theta_{DE}$ is equal to the subgroup
of $\pi_{E}(X)$ generated by all $E$-homotopy classes of $E$-refinements of
$D$-small loops.
\end{corollary}

\begin{proof}
It follows from Proposition \ref{deltaimp} that every element of $\ker
\theta_{ED}$ (i.e. an $E$-homotopy class of an $E$-loop that is $D$-null) is a
product of $E$-homotopy classes of $E$-refinements of $D$-small loops. On the
other hand, any concatenation of $E$-refinements of $D$-small loops is
$D$-null and hence its $E$-homotopy class is in $\ker\theta_{DE}$.
\end{proof}

Further extending the results of \cite{PW1} we define for any entourage $E$ in
a metric space $X$, a map from fixed-endpoint homotopy classes of continous
paths to $E$-homotopy classes of $E$-chains as follows. Suppose
$c:[0,1]\rightarrow X$ is continuous. We set $\Lambda_{E}([c]):=[\alpha]_{E}
$, where $\alpha$ is any $E$-subdivision of $c$. By Lemma \ref{subdiv},
$\Lambda_{E}$ is well-defined. Note that if $E$ is a chained entourage then by
Lemma \ref{strex} every $E$-chain $\alpha$ may be assumed, up to $E$-homotopy,
to have a stringing $\overline{\alpha}$. By definition $\Lambda_{E}%
(\overline{\alpha})=[\alpha]_{E}$; that is, $\Lambda_{E}$ is surjective.
Restricting $\Lambda_{E}$ to the fundamental group at any base point yields a
homomorphism $\pi_{1}(X)\rightarrow\pi_{E}(X)$ that we will also refer to as
$\Lambda_{E}$.

Continuing to assume that $E$ is a chained entourage, fix a basepoint and
suppose that $[c]\in\ker\Lambda_{E}$. In other words, any $E$-subdivision
$\{c(t_{0}),...,c(t_{n})=c(t_{0})\}:=\lambda$ is $E$-null. Taking $D=E$ in
Proposition \ref{deltaimp}, we see that $\lambda$ is $E$-homotopic to a
product of $E$-small $E$-loops. Since $\Lambda_{E}$ is a homomorphism, $c$ is
homotopic to the concatenation of stringings of $E$-small loops. We have shown:

\begin{theorem}
\label{ker}Let $X$ be a Peano continuum that has a (compatible) geodesic
metric and $E$ be a chained entourage. Then for any basepoint, $\Lambda
_{E}:\pi_{1}(X)\rightarrow\pi_{E}(X)$ is a surjective homomorphism and
$\ker\Lambda_{E}$ is the subgroup of $\pi_{1}(X)$ generated by homotopy
classes of stringings of $E$-small loops.
\end{theorem}

\begin{remark}
The above theorem actually only requires that $X$ have a geodesic metric
compatible with the uniform structure, see Remark \ref{pnc}.
\end{remark}

\begin{remark}
\label{coverstuff}Note that one may also restrict $\Lambda_{E}$ to the set of
equivalence classes of curves starting a fixed basepoint, which is by
definition the universal covering space of $X$ when $X$ is semi-locally simply
connected. According to Theorem 26.2 in \cite{PW1}, if $X$ is a compact,
semi-locally simply connected geodesic space, then $\Lambda_{E_{\varepsilon}%
}:\pi_{1}(X)\rightarrow\pi_{\varepsilon}(X)$ is length preserving when
$\varepsilon$ is a lower bound for HCS. That is, $\Lambda_{E_{\varepsilon}}$
restricts to a bijection from the universal cover $\widetilde{X}$ of $X$ to
$X_{\varepsilon}$. That is, we may identify the universal cover $\widetilde{X}%
$ of $X$ with $X_{\varepsilon}$ and $\pi_{1}(X)$ with $\pi_{\varepsilon}(X)$.
\end{remark}

\begin{remark}
In the above theorem we see a hint of the relationship between our
construction and the construction of Spanier used by Sormani-Wei, referred to
in the Introduction. For a metric entourage $E_{\varepsilon}$ in a geodesic
space, we may take stringings using geodesics, and such geodesics always
remain in $B(x_{i},\frac{3\varepsilon}{2})$. That is, $\ker\Lambda
_{E_{\varepsilon}}$ is precisely the Spanier subgroup used by Sormani-Wei, and
this is how the equivalence of $\varepsilon$-covers and $\delta$-covers was
proved in \cite{PW2}. One could take the Sormani-Wei use of Spanier a bit
farther by applying it to the covering of a space by entourage balls for a
fixed entourage, but it seems unlikely that the resulting covering maps would
not be equivalent to entourage covers as we have defined them.
\end{remark}

\begin{remark}
As for basepoints, as was observed in \cite{BPUU}, as long as $X$ is chain
connected, the various groups and homomorphisms defined above are independent
of the basepoint, up to natural isomorphisms induced by basepoint change. This
is why we have not included basepoints in our notation. When necessary we can
always assume that all mappings are basepoint-preserving (take basepoints to basepoints).
\end{remark}

\section{Properties of Entourage Covers}

Two chained entourages $E_{1},E_{2}$ in a uniform space will be called
\textit{equivalent if }$\phi_{E_{1}}$ and $\phi_{E_{2}}$ are equivalent as
covering maps.

\begin{proposition}
\label{cover}Let $X$ be a uniform space, $D\subset E$ be chained entourages,
$G$ be a normal subgroup of $\pi_{D}(X)$, and $\pi:X_{D}\rightarrow X_{D}/G=Y$
be the quotient covering map. Then there is a covering map $h:X_{E}\rightarrow
Y$ such that $h\circ\phi_{ED}=\pi$ if and only if $\ker\theta_{ED}\subset G$.
In particular, the covers $\phi_{E}:X_{E}\rightarrow X$ and the induced cover
$\phi:Y\rightarrow X$ are equivalent if and only if $G=\ker\theta_{ED}$.
\end{proposition}

\begin{proof}
If $\ker\theta_{ED}\subset G$ then as mentioned in the background section we
have $X_{E}=X_{D}/\ker\theta_{ED}$ and so $G/\ker\theta_{ED}$ acts properly
discontinuously on $X_{E}$ with quotient space naturally identified with
$X_{D}/G$ (cf. Theorem 1.6.11 in \cite{Sp}). That is, the quotient map%
\[
h:X_{E}=X_{D}/\ker\theta_{ED}\rightarrow X_{E}/(G/\ker\theta_{ED})=X_{D}/G=Y
\]
is a (regular) covering map that by definition satisfies $h\circ\phi_{ED}=\pi$.

Conversely, suppose that there is a covering map $h:X_{E}\rightarrow Y$ such
that $h\circ\phi_{ED}=\pi$. By composing with a covering equivalence we may
suppose that $h$ is basepoint preserving. Now suppose that $\theta
_{ED}([\lambda]_{D})=[\ast]_{E}$. Then since $h$ is basepoint preserving,
$h\circ\theta_{ED}([\lambda]_{D})=\pi([\ast]_{D})$. That is, $[\lambda]_{D}%
\in\pi^{-1}(\pi([\ast]_{E})=G$.
\end{proof}

\begin{corollary}
\label{equiv}Let $D$, $E$, $F$ be chained entourages in a uniform space $X$
with $D\subset E\cap F$. Then

\begin{enumerate}
\item $E$ and $F$ are equivalent if and only if $\ker\theta_{ED}=\ker
\theta_{FD}$.

\item There is a non-trivial covering map $h:X_{E}\rightarrow X_{F}$ if and
only if there is some $F$-triad with a $D$-refinement that is not $E$-null.
\end{enumerate}
\end{corollary}

\begin{proof}
The first statement is an obvious consequence of Proposition \ref{cover},
which also says that there is a non-trivial covering map $h:X_{E}\rightarrow
X_{F}$ if and only if $\ker\theta_{ED}$ is a proper subset of $\ker\theta
_{FD}$. Equivalently, there is an $D$-loop $\lambda$ that is $F$-null but not
$E$-null. Equivalently, by Proposition \ref{deltaimp},
\[
\lbrack\lambda]_{D}=[\lambda_{1}\ast\cdot\cdot\cdot\ast\lambda_{n}]_{D}%
\]
where each $\lambda_{i}$ is a $D$-refinement of an $F$-small loop, at least
one of which is not $E$-null. The proof is now finished by the definition of
$F$-small.
\end{proof}

\begin{remark}
Corollary \ref{equiv} can be considered as an extension of Corollary 31 in
\cite{BPUU} to include the situation when neither $E$ nor $F$ may be contained
in the other, but there is a covering equivalence between $\phi_{E}$ and
$\phi_{F}$.
\end{remark}

\begin{theorem}
\label{firstest}If $X$ is a geodesic space and $\varepsilon>0$ then there is a
set $S\subset\pi_{\varepsilon}(X)$ with $\left\vert S\right\vert \leq
C(X,\frac{\varepsilon}{4})^{40C(X,\frac{\varepsilon}{2})}$ such that if $E$ is
a chained entourage with $E_{\varepsilon}\subset E$ then $\ker\theta
_{EE_{\varepsilon}}$ is the normal closure of some subset of $\Gamma$.
\end{theorem}

\begin{proof}
Let $S$ be the set of all $[\beta]_{\varepsilon}\subset\pi_{\varepsilon}(X)$
such that $L(\beta)\leq10\varepsilon C(X,\frac{\varepsilon}{2})$. By Theorem
3.2 of \cite{PW1} $\left\vert S\right\vert \leq C(X,\frac{\varepsilon}%
{4})^{40C(X,\frac{\varepsilon}{2})}$. Therefore we need only prove that for
any such $E$, $\ker\theta_{EE_{\varepsilon}}$ is the normal closure of a set
of elements of length at most $10\varepsilon C(X,\frac{\varepsilon}{2})$. By
Corollary \ref{gens}, $\ker\theta_{EE_{\varepsilon}}$ is equal to the subgroup
of $\pi_{\varepsilon}(X)$ generated by the collection of all $[\alpha\ast
\tau\ast\overline{\alpha}]_{\varepsilon}$ where $\tau$ is an $E_{\varepsilon}%
$-refinement of an $E$-triad. But it is an easy algebraic argument that this
means that $\ker\theta_{EE_{\varepsilon}}$ is the normal closure of elements
of the form $\left[  \overline{\alpha_{\tau}}\ast\tau\ast\alpha_{\tau}\right]
_{\varepsilon}$, where (1) $\tau$ is an element of a set $\Gamma$ of
$E_{\varepsilon}$-refinements of $E$-triads that contains exactly one
representitive of each free $E_{\varepsilon}$-homotopy class of
$E_{\varepsilon}$-refinements of $E$-triads, and (2) $\alpha_{\tau}$ is any
$\varepsilon$-chain from the basepoint to the start/end point of $\tau$. In
fact, any generator of $\ker\theta_{EE_{\varepsilon}}$ is conjugate to some
such $\left[  \overline{\alpha_{\tau}}\ast\tau\ast\alpha_{\tau}\right]
_{\varepsilon}$.

According to Lemma \ref{conna} the chains $\alpha_{\tau}$ from the above
paragraph may be chosen to have at most $2C(X,\frac{\varepsilon}{2})$ points.
Likewise, we may produce an $E_{\varepsilon}$-subdivision of $\tau$ having at
most $6C(X,\frac{\varepsilon}{2})$, and therefore there are $E_{\varepsilon}%
$-refinements of the $E$-small loops $\overline{\alpha_{\tau}}\ast\tau
\ast\alpha_{\tau}$ having at most $10C(X,\frac{\varepsilon}{2})$ points and
hence length at most $10\varepsilon C(X,\frac{\varepsilon}{2})$.
\end{proof}

Theorem \ref{T2} now follows from Theorem \ref{firstest} and Proposition
\ref{equiv}.

To prove Theorem \ref{csecs}, suppose that $\sigma$ is any value of CS. By
definition the corresponding $\sigma$-cover is not equivalent to any $\delta
$-cover for $\delta>\sigma$. Now suppose that $E$ is an entourage that
contains $E_{\delta}$ for some $\delta>\sigma$. Then $\phi_{EE_{\sigma}}%
=\phi_{E_{\delta}E}\circ\phi_{E_{\delta}E_{\sigma}}$ and since $\phi
_{E_{\delta}E_{\sigma}}$ is not injective (and all maps are surjective),
$\phi_{EE_{\sigma}}$ is not either. That is, there is at least one chained
entourage (namely $E_{\sigma}$) that contains $E_{\sigma} $ and is not
equivalent to any $E$ with $\sigma(E)>\sigma$, hence $\sigma\in$ ECS.

To prove Theorem \ref{obstruct} we may assume that $X$ is a geodesic space by
the Bing-Moise Theorem. Since $X$ is semilocally simply connected, Corollary
43 of \cite{PW1} implies that for all sufficiently small $\varepsilon$,
$\phi_{\varepsilon}:X_{\varepsilon}\rightarrow X$ is the universal covering
map of $X$ and $\pi_{\varepsilon}(X)=\pi_{1}(X)$. By definition, $\phi_{E}$ is
the covering map corresponding to $\ker_{EE_{\varepsilon}}$ and the proof of
Theorem \ref{obstruct} is finished by Theorem \ref{firstest}.

\begin{remark}
In \cite{BPUU}, Theorem 37, Berestovskii-Plaut showed that for any entourage
$E$ in a compact uniform space, if $\phi_{E}$ is chain connected (which is
true when $E$ is chained according to our new definition) then $\pi_{E}(X)$ is
finitely generated. Essentially the same argument as the proof of Theorem
\ref{obstruct}, together with Remark \ref{finpres}, shows that if $X$ is a
Peano continuum and $E$ is a chained entourage then $\pi_{E}(X)$ is in fact
finitely presented.
\end{remark}

For the proof of Theorem \ref{TD3} we need to some facts from \cite{PW1} and
\cite{PW2} concerning essential circles, along with a new result (Proposition
\ref{systole}). Essential circles are defined (Definition 5, \cite{PW1}) to be
continuous paths of length $L$ that are not $\varepsilon$-null for
$\varepsilon=\frac{L}{3}$ (we may also refer to it as an essential
$\varepsilon$-circle). According to Lemma 33, \cite{PW1} this means that
essential circles are characterized as curves of positive length that are
shortest in their $\varepsilon$-homotopy class for some $\varepsilon$.
Essential circles are special closed geodesics that are miminal on half their
length (\textquotedblleft$2$-geodesics\textquotedblright\ in the parlance of
Sormani-Wei, $\frac{L}{2}$-geodesics in our terminology) whose lengths are
three times the lengths of the values in HCS, or equivalently $\frac{2}{3}$
the values of CS (\cite{PW1}, Theorem 6). If $C$ is an essential $\varepsilon
$-circle then any triple of points $T=\{x_{0},x_{1},x_{2}\}$ on $C$ such that
$d(x_{i},x_{j})=\varepsilon$ when $i\neq j$ is an \textit{essential
}$\varepsilon$-\textit{triad }(or just \textit{essential triad} when
$\varepsilon$ is unspecified), meaning that no $\varepsilon$-subdivision of
$T$ is $\varepsilon$-null. Essential triads are characterized by the fact that
if the points on them are joined by geodesics the resulting curve is an
essential circle. To summarize, essential triads are precisely the discrete
analogs of essential circles: adding \textquotedblleft edges\textquotedblright%
\ to an essential triad creates an essential circle, and any triad of equally
spaced points on an essential circle is an essential triad.

Two essential triads $\tau_{1},\tau_{2}$ are said to be \textit{equivalent} if
they are both essential $\varepsilon$-triads for some $\varepsilon>0$, and
some, hence any $\varepsilon$-subdivision of $\tau_{1}$ is $\varepsilon
$-homotopic to an $\varepsilon$-subdivision of $\tau_{2}$ or $\overline
{\tau_{2}}$. Two essential circles are said to be \textit{equivalent} if the
corresponding essential triads are equivalent (see \cite{PW1} for more
details). As we will now show, the smallest essential circles are generally
easiest to find.

\begin{proposition}
\label{systole}Let $X$ be a compact geodesic space that is semilocally simply
connected. If $c$ is a path loop that is not null-homotopic and whose length
is equal to the $1$-systole $\sigma_{1}$ of $X$ (i.e. the length of the
shortest curve that is not null-homotopic) then $c$ a shortest essential
circle. Moreover, any two shortest essential circles are equivalent if and
only if one is freely homotopic to the other or its reversal.
\end{proposition}

\begin{proof}
The fact $\sigma_{1}$ is positive and is the shortest possible length of any
essential circle is part of Corollary 43 in \cite{PW1}. It follows from
Theorem 6 in \cite{PW1} that there is at least one essential circle of length
$\sigma_{1}$. Now suppose that $c$ is not null-homotopic of length $\sigma
_{1}$. Taking $\varepsilon:=\frac{\sigma_{1}}{3}$, to show $c$ is an essential
circle we need only show that it is not $\varepsilon$-null. Suppose it were;
that is, some $\varepsilon$-subdivision $\lambda$ of $c$ must be $\varepsilon
$-null. Taking $\delta=\varepsilon$ in Proposition 30 of \cite{PW1} (which is
analgous to Proposition \ref{deltaimp} in the present paper), for some choice
of $\alpha$, $\overline{\alpha}\ast\lambda\ast\alpha$ is $\varepsilon
$-homotopic to a product of $\varepsilon$-small $\varepsilon$-loops, one of
which must not be $\varepsilon$-null. But recall that an $\varepsilon$-small
loop is of the form $\overline{\beta}\ast\tau\ast\beta$, where $\tau
=\{x_{0},x_{1},x_{2},x_{0}\}$ satisfies $d(x_{i},x_{j})<\varepsilon$ for all
$i,j$. But such a $\tau$ is necessarily $\varepsilon$-null, a contradiction.

Next, suppose that $c_{1},c_{2}$ are equivalent shortest essential circles.
Reversing $c_{2}$, if necessary, this means that for $\varepsilon=\frac
{\sigma_{1}}{3}$, some essential triad $\tau_{i}$ on $c_{i}$ has an
$\varepsilon$-subdivision $\tau_{i}^{\prime}$ that is $\varepsilon$-homotopic
to $\tau_{j}^{\prime}$ for $j\neq i$. That is, $c_{1}$ is $\varepsilon
$-homotopic to $c_{2}$. By Theorem 26 of \cite{PW1}, $X_{\varepsilon}$ is the
universal covering space of $X$, meaning that two loops lift to a loop in
$X_{\varepsilon}$ if and only if they are homotopic. But the same statement is
true for $\varepsilon$-homotopic curves by Proposition \ref{ender}, completing
the proof.
\end{proof}

In \cite{PW2}, Theorem 27, when $0<\delta<\varepsilon$, we showed that
$\phi_{\varepsilon\delta}:X_{\delta}\rightarrow X_{\varepsilon}$ is
characterized as the quotient map of $X_{\delta}$ via the (normal) subgroup
$K_{\varepsilon}(\mathcal{T})$ of $\pi_{\varepsilon}(X)$ generated by all
$\varepsilon$-loops of the form $\overline{\alpha}\ast\tau_{i}\ast\alpha$,
where is in a set $\mathcal{T=\{}\tau_{1},...,\tau_{k}\}$ containing exactly
one essential triad $\tau_{i}$ representing each equivalence class of
essential $\mu$-triads with $\varepsilon\leq\mu<\delta$. Via the argument in
the proof of Theorem \ref{firstest}, $K_{\varepsilon}(\mathcal{T})$ is in fact
the normal closure of the finite set of all $\alpha_{i}\ast\tau_{i}%
\ast\overline{\alpha_{i}}$ for any specific choice of $\varepsilon$-chain
$\alpha_{i}$.

\begin{proof}
[Proof of Theorem \ref{TD3}]In \cite{GSS3}, DeSmit, Gornet, and Sutton
introduced (Definition 2.3) the notion of a \textit{length map} on a group $H$
with identity $1$, namely a function $m:H\rightarrow R^{+}$ such that (a) $m$
is positive except $m(1)=0$, (b) $m(hgh^{-1})=m(g)$ for all $g,h\in H$, and
(c) $m(g^{k})\leq\left\vert k\right\vert m(g)$ for all $g\in H$ and
$k\in\mathbb{Z}$. A particular example of a length map is what Sormani-Wei
(\cite{SW1}) called the Minimum Marked Length Map $m_{g}$ in on a manifold
with Riemannian metric $g$: $m_{g}$ assigns to each element of the fundamental
group the length of the shortest curve its free homotopy class. So the values
of $m_{g} $ are precisely MLS. We will need Theorem 2.9 of \cite{GSS3}, which
we will describe in a weaker simplified form using Example 5.5 \cite{GSS3},
and we will shift the subscripts to improve the exposition for our purposes.
Let $M$ be a Riemannian manifold of dimension at least $3$. Let $\mathcal{F}%
(M)$ denote the collection of un-oriented free homotopy classes of loops in
$M$. Suppose that $\{c_{0},...,c_{k}\}$ is a set of distinct elements of
$\mathcal{F}(M)$ with $c_{0}$ trivial, and $l_{0}=0<l_{1}\leq\cdot\cdot
\cdot\leq l_{k}$ is a sequence of real numbers such that $2l_{1}\geq l_{k}$.
Then by Theorem 2.9 and Example 5.5 of \cite{GSS3} there is a Riemannian
metric $g$ on $M$ such that $m_{g}(c_{i})=l_{i}$ for all $i$ and $m_{g}(c)\geq
l_{k}$ otherwise. In other words, one may prescribe the length of the shortest
geodesic in the free homotopy class of all $c_{i}$, and force all other values
of MLS to be at least $l_{k}$.

Suppose that $G$ is the normal closure of a finite set $\{g_{1},...,g_{k-1}\}
$, all distinct and none of which is trivial. If $G$ is equal to $\pi_{1}(M) $
then the corresponding covering map is trivial, hence an $\varepsilon$-cover
for any Riemannian metric. Assume there is some nontrivial $g_{k}\notin G$.
Now let $c_{0}$ be the trivial free homotopy class and for each $1\leq i\leq
k$ let $c_{i}$ be the free homotopy class of some, hence any loop in $g_{i}$.
Define $l_{0}:=0$, $l_{k}=1.5$ and $l_{i}=1$ for $i=1,...,k-1$. By Example 5.5
in \cite{GSS3} (noting that our indexing begins with $0$ rather than $1$) and
Theorem 2.9 of \cite{GSS3} there is Riemannian metric $g$ on $M$ such that
$m_{g}(c_{i})=1$ for all $i\neq0,k$ and $m_{g}(c)\geq1.5$ for every free
homotopy class $c\neq c_{i}$ for any $i<k$. Let $\kappa_{i}$ be shortest
curves representing each $c_{i}$ with $1\leq i\leq k-1$. By Proposition
\ref{systole}, each $\kappa_{i}$ is a shortest essential circle, and two
$\kappa_{i}$ are non-equivalent as essential circles if and only if they lie
in different $c_{i}$. Moreover, any other essential circle, being shortest in
its homotopy class, must have length at least $1.5$.

As discussed in Remark \ref{coverstuff}, we may suppose that $\varepsilon>0$
is small enough that we may identify, via the function $\Lambda
_{E_{\varepsilon}}$, $M_{\varepsilon}$ with the universal cover $\widetilde{M}%
$ of $M$ and identify $G$ with a normal subgroup of $\pi_{\varepsilon}%
(M)=\pi_{1}(M)$. Under this correspondence, $G$ is the normal closure of
$\{[\lambda_{i}]_{\varepsilon}\}$, where $\lambda_{i}$ is any $E_{\varepsilon
}$-subdivision of a loop of the form $\overline{f_{i}}\ast c_{i}\ast f_{i}$,
where $f_{i}$ is any curve from the basepoint to the start point of the
essential circle $\kappa_{i}$. That is, $G$ is precisely equal to $K_{\sigma
}(\mathcal{T)}$ (see discussion prior to this proof), where $\mathcal{T}%
:=\{\tau_{i}\}_{i=1}^{k-1}$ is a collection of essential triads representing
all equivalence classes of smallest essential circles, with exactly one
representative $\tau_{i}$ for each free homotopy class. On the other hand,
letting $\delta:=\frac{1.5}{3}=\frac{1}{2}$, Theorem 27 of \cite{PW2} states
that since $\mathcal{T}$ contains a representative for each essential $\sigma
$-triad with $\frac{1}{2}<\sigma<\varepsilon$, $K_{\varepsilon}(\mathcal{T}%
)=\ker\theta_{\delta\varepsilon}$. That the covering space $\widetilde{M}/G$
is equivalent to $M_{\delta}$ now follows from Proposition \ref{cover}.
\end{proof}

\begin{proof}
[Proof of Proposition \ref{semi}]The statements about inessential $E$ are
obvious. The proof of the rest is very similar to the proof of the statement
for Riemannian manifolds involving free homotopy classes, replacing the fact
that small loops are null-homotopic by the fact that small loops are $E$-null.
Suppose that $E$ is essential, $c$ is not freely $E$-null and let
$c_{i}:[0,1]\rightarrow X$ be $E$-loops that are freely $E$-homotopic to $c$,
parameterized proportional to arclength, with lengths converging to
\[
L:=\inf\{L(f):f\text{ is a curve loop that is }E\text{-homotopic to
}c\}\text{.}%
\]
Since the lengths of the $c_{i}$ are bounded, a standard application of the
Ascoli-Arzela Theorem shows we may assume, taking a subsequence if needed,
that $c_{i}$ converges uniformly to some $\overline{c}:[0,1]\rightarrow X$
with $L(c)\leq L$. By Proposition \ref{uniform}, for all large $i$,
$\overline{c}$ is freely $E$-homotopic to $c_{i}$ hence to $c$. So
$\overline{c}$ is the desired shortest curve. To see that any such
$\overline{c}$ is a closed $\frac{3\varepsilon}{2}$-geodesic, suppose that a
segment $S$ of length $\leq\frac{3\varepsilon}{2}$ is not minimal. By
definition we may join its endpoints by a new curve $S^{\prime}$ of length
less than $\frac{3\varepsilon}{2}$. Then $S$ and $S^{\prime}$ together form a
loop of length less than $3\varepsilon$, which is $\varepsilon$-null (Lemma
33, \cite{PW1}), hence $E$-null. Therefore the curve obtained by replacing $S$
by $S^{\prime}$ has length shorter than $L$ and is $E$-homotopic to $c$, a contradiction.
\end{proof}

\begin{proof}
[Proof of Theorem \ref{final}]In any geodesic space, if $c$ is shortest in its
$E$-homotopy class, it is shortest in its homotopy class by Corollary
\ref{homimp}, showing that the set described in Part 1a is contained in MLS.
On the other hand, suppose that $c$ is shortest in its homotopy class. As in
the proof of Proposition \ref{systole}, let $\varepsilon>0$ be small enough
that $X_{\varepsilon}$ is the universal covering space of $X$, so that curves
are freely homotopic if and only if they are freely $\varepsilon$-homotopic.
That is, $c$ is shortest in its $\varepsilon$-homotopy class, i.e. its
$E$-homotopy class for $E=E_{\varepsilon}$. This proves Part 1a. The remaining
parts are true for any compact geodesic space, as will be shown next.
\end{proof}

\begin{theorem}
\label{final'}Theorem \ref{final} holds when $M$ is simply assumed to be a
compact geodesic space, using Theorem \ref{final}.1a as the definition of MLS.
\end{theorem}

\begin{proof}
If $c$ is non-constant and shortest in its free $E$-homotopy class then
according to Proposition \ref{finally2}.1 there is some $E$-loop $\lambda$
having the same length as $c$ that is $E$-homotopic to any $E$-subdivision of
$c$. But Proposition \ref{finally2}.2 now shows that $\lambda$ must also be
shortest in its $E$-homotopy class. An analogous argument shows that if
$\lambda^{\prime}$ is shortest in its $E$-homotopy class then there must be a
curve of the same length as $\lambda^{\prime}$ that is shortest in its $E
$-homotopy class. That is, the quantities described by Parts 1a and 1b in
Theorem \ref{final'} are the same.

For the second part, note that in \cite{GSS3}, Section 3, de Smit, Gornet, and
Sutton gave the following equivalent definition of CS in any compact geodesic
space $X$: CS consists of half the lengths of loops that lift as a non-loop to
\textit{any} covering space of $X$. For the covering spaces $X_{E}$, Lemma
\ref{ehomo} implies that this is precisely the length of a shortest loop that
is not $E$-null, so any value in Part 2a is contained in CS. On the other
hand, CS $=\frac{3}{2}$ HCS consists of $\frac{3}{2}$ the lengths of all
essential circles. Since every essential $\varepsilon$-circle has length
$3\varepsilon$ and curves of length less than $3\varepsilon$ are $\varepsilon
$-null (Lemma 33, \cite{PW1}), then essential circles are the shortest
possible loops that are not $E$-null for $E=E_{\varepsilon}$. This completes
the proof that CS is consists of the values in Part 2a.

Note that if $c$ is a shortest loop that is not $E$-null then $c$ must be
shortest in its $E$-homotopy class. Proposition \ref{finally2}.1 tells us that
there is an $E$-loop of the same length as $c$ that is $E$-homotopic to any
$E$-subdivision of $c$. By Proposition \ref{finally2}.2, there cannot be a
shorter such loop. This justifies the term \textquotedblleft
shortest\textquotedblright\ in 2b. Moreover, as in the proof of Part 1,
Proposition \ref{finally2} tells us that the values in 2a and 2b are the same.
\end{proof}

\begin{proof}
[Proof of Theorem \ref{final2}]For the first part, note that if $c$ is
$E$-critical then any $E$-subdivision $\lambda$ of $c$, being also an
$\overline{E}$-subdivision of $c$, is also $E$-critical. Likewise, any
stringing $\widehat{\lambda}$ of an $E$-critical $E$-loop $\lambda$ is also a
stringing of $\lambda$ when considered as an $\overline{E}$-chain. It is
immediate that $\widehat{\lambda}$ is an $E$-critical loop.

The proof of the second part is similar to the proof of Proposition \ref{semi}
and we will only state the essential steps for the argument involving $c$. Let
$c_{i}:[0,1]\rightarrow X$ be $E$-loops that are $E$-critical, parameterized
proportional to arclength, with lengths converging to $\psi(E) $. We may
assume that $c_{i}$ converges uniformly to some $c:[0,1]\rightarrow X$ with
$L(c)\leq\psi(E)$. For all large $i$, $c$ is freely $E$-homotopic and
$\overline{E}$-homotopic to $c_{i}$. That is, $c$ is also $E$-critical and has
shortest length. Proposition \ref{final2} now finishes the proof.

For the third part, recall that HCS consists of $\frac{1}{3}$ the lengths of
essential triads. But any essential $\varepsilon$-triad $T$, having length
$3\varepsilon$, is $\overline{E_{\varepsilon}}$-null, but by defininition has
no $\varepsilon$-null $\varepsilon$-subdivision. That is, any $\varepsilon
$-subdivision of $T$ (which also has length $3\varepsilon$) is $E_{\varepsilon
}$-critical. By definition, CS $\subset$ ES. On the other hand, if $c$ is a
shortest $E$-critical curve then $c$ must be shortest in its free $E$-homotopy
class. If not, it would be $E$-homotopic, hence $\overline{E}$-homotopic to a
shorter curve--but that curve would still be $E$-critical, a contradiction.
The final two statements are shown in Example \ref{distinct}.
\end{proof}

\section{Examples}

\begin{example}
\label{circle}Let $X$ be a circle with the unique geodesic metric of
circumference $1$ and let $E$ be a chained entourage. According to Theorem
\ref{ker}, $\ker\Lambda_{E}$ is generated by homotopy classes of stringings of
$E$-small loops. We will show that every $E$-triad has a stringing that
represents either the trivial homotopy class or the class of a generator of
$\pi_{1}(X)=\mathbb{Z}$. From this it follows that the only two possible
$E$-covers of the circle are the trivial cover and the universal cover.

If an $E$-triad $\tau=\{x_{0},x_{1},x_{2}\}$ has no $E$-null stringing then in
particular all three points must be distinct. We assume that the points are
ordered in the clockwise direction and let $A_{i}$ denote the arc in the
clockwise direction from $x_{i}$ to $x_{i+1}$. Since $E$ is chained, each
$B(x_{i},E)\cap B(x_{i+1})$ must contain $A_{i}$ or $A_{j}\cup A_{k}$ with
$j,k\neq i$. If $B(x_{i},E)\cap B(x_{i+1})$ only contains $A_{i}$ for all $i$,
then the only possible stringings of $\tau$ are $E$-homotopic to the circle
itself, meaning that the $\phi_{E}$ is trivial.

Otherwise, without loss of generality we may suppose that $B(x_{0},E)\cap
B(x_{1},E)$ contains $A_{1}\cup A_{2}$. Case 1: Suppose that $A_{2}\subset
B(x_{2},E)\cap B(x_{0},E)$. There are two subcases. 1a: $A_{1}\subset
B(x_{1},E)\cap B(x_{2},E)$. In this case (slightly abusing notation by
considering each $A_{i}$ as a path), $\overline{A_{2}}\ast\overline{A_{1}}\ast
A_{1}\ast A_{2}$ is a stringing of $\tau$ that is clearly $E$-null. 1b:
$A_{0}\cup A_{2}\subset B(x_{1},E)\cap B(x_{2},E)$. In this case,
$\overline{A_{2}}\ast\overline{A_{1}}\ast\overline{A_{0}}\ast\overline{A_{2}%
}\ast A_{2}$ is a stringing of $\tau$ that represents a generator of $\pi
_{1}(X)$. Now observe that we have considered the three essential cases (up to
re-ordering): Each $B(x_{i},E)\cap B(x_{i+1})$ contains only $A_{i}$, exactly
two of the $B(x_{i},E)\cap B(x_{i+1})$ contain only $A_{i}$, or exactly one of
the $B(x_{i},E)\cap B(x_{i+1})$ contains only $A_{i}$, so the proof is complete.
\end{example}

The next examples are related to the question of identifying entourage covers
for $2$-dimensional manifolds (with or without boundary).

\begin{example}
\label{moebius}Let $M$ be the Moebius Band. First note that $M$ is
homeomorphic to $\mathbb{RP}^{2}$ with a small disk removed. Therefore we may
Gromov-Hausdorff approximate $\mathbb{RP}^{2}$ by Mobius bands, taking the
standard metric on $\mathbb{RP}^{2}$ and induced geodesic metric on $M$ with
smaller and smaller disks removed. Since the double cover of $\mathbb{RP}^{2}$
is its universal cover, hence an $\varepsilon$-cover for all sufficiently
small $\varepsilon$, the double cover of $M$ is also an $\varepsilon$-cover
for small enough $\varepsilon$ with respect to the induced metrics with small
enough disks removed. This follows from the convergence result Proposition 37,
\cite{PW2}--since small enough $\varepsilon>0$ is not a homotopy critical
value of $\mathbb{RP}^{2}$. Note that $M$ deformation retracts onto the circle
and yet the double cover of the circle is not an entourage cover by Example
\ref{circle}. This example shows one way that convergence can be used to
identify entourage covers but also shows the limitations of \textquotedblleft
enlarging\textquotedblright\ spaces to try to find entourage covers.
\end{example}

\begin{remark}
Let $X$ be a compact geodesic space such that HCS has $n$ elements, counting
multiplicity as defined in \cite{PW1}. In other words, there are a total of
$n$ equivalence classes of essential circles. As pointed out in the proof of
Theorem \ref{final2}, every essential $\varepsilon$-circle is $E_{\varepsilon
}$-critical, and as defined in \cite{PW1}, the multiplicity is the number of
distinct $E_{\varepsilon}$-homotopy classes of $\varepsilon$-circles. Since
being $E$-critical is a topological property these essential circles will
still be critical loops in any other metric on $X$, and by definition their
lengths will be elements of ES. That is, ES in any metric will always have at
least $n$ elements, but the lengths of curves, and hence the values and
multiplicities of ES will be different. In particular, as soon as there are
two geodesic metrics with CS having different sizes, then the space with the
smaller CS must have CS strictly contained in ES, counting multiplicity. In
the next example we essentially carry this out in a carefully controlled
setting, allowing us to insure that the multiplicities are 1 and hence control
the absolute size of the spectra.
\end{remark}

\begin{example}
\label{distinct}Let $M$ be a compact smooth manifold of dimension $3$ or
higher with fundamental group $\mathbb{Z}_{4}$, which we will denote by
$\{[c_{0}],[c_{1}],[c_{2}],[c_{3}]\}$. (There are many other possibilities for
$\pi_{1}(M)$, but we are using $\mathbb{Z}_{4}$ for simplicity.) Denote the
double cover of $M$ (corresponding to the subgroup generated by $[c_{2}]$) by
$M^{\prime}$ and the universal cover of $M$ by $M^{\prime\prime}$. By standard
covering space theory, $c_{2}$ is the only loop, up to free homotopy, that
lifts as a loop to $M^{\prime}$. We now make the following assignments:
$[c_{1}]\rightarrow1.2$, $[c_{2}]\rightarrow1.1$, $[c_{3}]\rightarrow1.2$ and
apply Theorem 2.9 of \cite{GSS3}. In the resulting metric, the shortest loop
lifting to a non-loop in $M^{\prime}$ must be freely homotopic to either
$c_{1}$ or $c_{3}$ and hence has length $1.2$. According to Proposition
\ref{systole}, $M^{\prime\prime}$ is equivalent to $M_{\frac{1.1}{3}}$. In
particular, HCS $=\{\frac{1.1}{3},\frac{1.2}{3}\}$ and CS $=\{\frac{1.1}%
{2},\frac{1.2}{2}\}$. Moreover, as pointed out in the proof of Theorem
\ref{final2}, if $c$ is an essential $\frac{1.2}{3}$-circle then $c$ is
$E$-critical, where $E:=E_{\frac{1.2}{3}}$. Note that since $c$ lifts as a
non-loop to $M_{E}$, $c$ is freely homotopic to $c_{2}$.

Now use the following assignments: $[c_{1}]\rightarrow1.1$, $[c_{2}%
]\rightarrow1.2$, $[c_{3}]\rightarrow1.3$ and apply Theorem 2.9 of
\cite{GSS3}. With this metric, any shortest loop lifting as a non-loop to
either $M^{\prime}$ and $M^{\prime\prime}$ has length $1.1$.That is, HCS
$=\{\frac{1.1}{3}\}$ and CS $=\{\frac{1.1}{2}\}$. Since being $E$-critical is
a topological property, the loop $c$ is still $E$-critical and by definition
the length of the shortest curve $c^{\prime}$ in its free $E$-homotopy class
is an element of ES. But any curve in the free $E$-homotopy class of $c $ must
lift to a loop in $M^{\prime}=M_{E}$. Since $c_{2}$ is the only such curve, up
to free homotopy, $c^{\prime}$ must be freely homotopic to $c_{2}$ and
therefore has length $1.2$. That is, ES contains $\{1.1,1.2\}$ and in
particular ES strictly contains $3$HCS $=2$CS. Moreover, since we may change
the value assigned to $[c_{2}]$ by any small amount, it is possible for
different Riemannian metrics on $M$ to have the same CS but different ES.
Finally, since there are only two non-trivial entourage covers of $M$, hence
at most two elements in ES, ES $=\{1.1,1.2\}$. But $1.3$ is an element of MLS
by definition, so ES is strictly contained in MLS. Again, the value assigned
to $[c_{3}]$ may be changed by a small amount without impacting ES, and
therefore one obtains Riemannian metrics on $M$ that have the same ES but
different MLS. This example verifies Theorem \ref{final2}.2.d-e.
\end{example}

\begin{example}
\label{nomin}In \cite{BPS}, Berestovskii-Plaut-Stallman gave two examples of
compact geodesic spaces with free homotopy classes having no closed geodesics
in them. The first is 1-dimensional, consisting of a circle in the plane with
smaller and smaller straight segments added to join points around the
perimeter circle, with the induced geodesic metric. The circle itself is not a
closed geodesic because any segment may be \textquotedblleft
bypassed\textquotedblright\ by a shorter straight segment, but it is, up to
monotone reparametization, the only curve, hence the shortest curve, in its
homotopy class. Now for any entourage $E$, the circle is $E$-homotopic to a
shorter curve that is a piecewise segment and is shortest in its $E$-homotopy
class. Note that although this segment visually has \textquotedblleft
corners\textquotedblright\ in the construction, it is still a closed geodesic
with the induced geodesic metric.

The other example is the infinite torus, the countable product $T^{\infty}$ of
circles with the Tychonoff topology. This space can be metrized by making the
sizes of the circles square summable and taking the geometric product metric
(see \cite{PS}). This space is of course a compact topological group and this
metric is bi-invariant. For this example, Stallman proved in his dissertation
that (1) there is a unique 1-parameter subgroup (i.e. a homomorphism
$\theta:\mathbb{R\rightarrow}T^{\infty}$) in each free homotopy class of a
curve, (2) there are free homotopy classes containing no rectifiable curves
(3) if there are rectifiable curves then the shortest one is the unique
1-parameter subgroup in it, but (4) the 1-parameter subgroup may not be a
closed geodesic. It would be interesting to see whether the shortest curves in
free $E$-homotopy classes given by Proposition \ref{semi} are also 1-parameter subgroups.
\end{example}

\end{document}